\documentclass[12pt,a4paper]{amsart}

\usepackage[nocompress,noadjust]{cite}

\usepackage{listings}
\usepackage{docmute}
\usepackage{amssymb, enumerate, calligra}
\usepackage{amsthm}
\usepackage{amsmath}
\usepackage{amsxtra}
\usepackage{latexsym}
\usepackage{mathrsfs}
\usepackage[all,cmtip]{xy}
\usepackage[all]{xy}
\usepackage{enumitem}
\usepackage{xcolor}
\usepackage{graphicx}
\usepackage{comment}
\usepackage{mathabx,epsfig}

\usepackage{hyperref}
\usepackage[nameinlink]{cleveref}
\usepackage{autonum}

\usepackage{mathtools}

\newtheorem{thm}{Theorem}[section]
\crefname{thm}{Theorem}{Theorems}
\newtheorem{lem}[thm]{Lemma}
\crefname{lem}{Lemma}{Lemmas}
\newtheorem{conj}[thm]{Conjecture}
\newtheorem*{conj*}{Conjecture}
\crefname{conj}{Conjecture}{Conjectures}

\crefname{claim}{Claim}{Claims}
\newtheorem{prop}[thm]{Proposition}
\crefname{prop}{Proposition}{Propositions}
\newtheorem{cor}[thm]{Corollary}
\crefname{cor}{Corollary}{Corollaries}

\crefname{property}{Property}{Properties}

\crefname{que}{Question}{Questions}

\theoremstyle{definition}
\newtheorem{defn}[thm]{Definition}
\crefname{defn}{Definition}{Definitions}
\newtheorem{rmk}[thm]{Remark}
\crefname{rmk}{Remark}{Remarks}
\newtheorem*{ack}{Acknowledgements}
\numberwithin{equation}{section}

\crefname{ex}{Example}{Examples}

\crefformat{equation}{{\upshape(#2#1#3)}}
\numberwithin{equation}{section}

\def\C{{\mathbb C}}
\def\Q{{\mathbb Q}}
\def\R{{\mathbb R}}
\def\Z{{\mathbb Z}}
\def\P{{\mathbb P}}
\def\A{{\mathbb A}}

\def\Gm{{\mathbb{G}}_{{\mathrm m}}}

\def\QQ{\overline{\mathbb Q}}
\def\O{{ \mathcal{O}}}

\def\n{{ \mathfrak{n}}}

\def\E{{ \mathcal{E}}}

\DeclareMathOperator{\pr}{pr}
\DeclareMathOperator{\Pic}{Pic}

\DeclareMathOperator{\Spec}{Spec}
\DeclareMathOperator{\Proj}{Proj}

\DeclareMathOperator{\Supp}{Supp}
\DeclareMathOperator{\Sym}{Sym}

\DeclareMathOperator{\Nef}{Nef}

\DeclareMathOperator{\Eff}{Eff}

\renewcommand{\d}{\delta}

\newcommand{\f}{\varphi}

\renewcommand{\k}{\kappa}

\newcommand{\G}{\Gamma}

\newcommand{\hh}{\hat{h}}

  {\end{list}}

\title[Kawaguchi-Silverman conjecture]
{Recent advances on Kawaguchi-Silverman conjecture}
\author{Yohsuke Matsuzawa}

\address{Department of Mathematics, Graduate School of Science, Osaka Metropolitan University, 3-3-138, Sugimoto, Sumiyoshi, Osaka, 558-8585, Japan}

\email{\href{mailto:matsuzaway@omu.ac.jp}{matsuzaway@omu.ac.jp}}

\begin{document}

\begin{abstract}
This is a survey on Kawaguchi-Silverman conjecture.
\end{abstract}

\maketitle

\setcounter{tocdepth}{2}

\tableofcontents

\section{Introduction}

The Kawaguchi-Silverman conjecture predicts the equivalence of the geometric complexity 
of self-maps and the arithmetic complexity of orbits of rational points.
Observations essential to the conjecture were already made in the 1990s, 
but the conjecture itself was explicitly formulated and stated around 2011 by Kawaguchi and Silverman, 
building on their earlier work. 
Following its formulation, people began to prove many special cases of the conjecture 
and study arithmetic degrees, a key concept within the conjecture. 
In this survey, we introduce the conjecture, outline some basic observations, and collect results that affirm special cases of the conjecture.
We have done our best to cover all results directly related to the conjecture, but we might have missed a few things.

\vspace{5mm}

{\bf Notation.}

\begin{enumerate}
\item
A variety over a field $k$ is (unless otherwise specified) a separated irreducible reduced scheme of finite type over $k$.

\item
Let $X$ be a projective variety over an algebraically closed field.
For Cartier divisors (resp.\ $\Q$ or $\R$-Cartier divisors) $D, E$ on $X$,
the linear equivalence (resp.\ $\Q$ or $\R$-linear equivalence) is denoted by
$D \sim E$ (resp.\ $D \sim_{\Q} E$ or $D \sim_{\R} E$).
They are said to be numerically equivalent if $(D\cdot C) = (E \cdot C)$ for all irreducible curves $C \subset X$.
In this case, we write $D \equiv E$.
The group of Cartier divisors modulo numerical equivalence is denoted by $N^{1}(X)$.
We write $N^{1}(X)_{\Q} = N^{1}(X) {\otimes}_{\Z} \Q$ and $N^{1}(X)_{\R} = N^{1}(X) {\otimes}_{\Z} \R$.

\item
The Kodaira dimension of a variety $X$ is denoted by $\k(X)$.
For a Cartier divisor $D$ on a projective variety $X$, 
its Iitaka dimension is denoted by $\k(X, D)$.

\item
For a dominant rational map $f \colon X \dashrightarrow X$ on a projective variety $X$ defined over an algebraically closed field
of characteristic zero, the $i$-th dynamical degree of $f$ is denoted by $\d_{i}(f)$ ($0 \leq i \leq \dim X$).

\item
For an $\R$-Cartier divisor $D$ on a projective variety $X$ defined over $\QQ$,
$h_{D}$ denotes a logarithmic Weil height function associated with $D$. 
We write $h_{D}^{+} = \max\{1, h_{D}\}$.

\end{enumerate}

\section{Kawaguchi-Silverman Conjecture}

\subsection{Arithmetic degree}

Let  $X$ be a quasi-projective variety defined over an algebraic closure of rational numbers $\QQ$.
Let us consider a dominant rational self-map $f \colon X \dashrightarrow X$ defined over $\QQ$.
As $f$ is a rational map and not necessarily a morphism, there might be points at which $f$ is not defined.
The smallest closed subset $I_{f} \subset X$ such that $f$ is a morphism on $X \setminus I_{f}$ is called the indeterminacy locus of $f$.
To discuss the dynamics of $f$, let us introduce the following notation.

\begin{defn}
The set of $\QQ$-points whose forward $f$-orbit is well defined is denoted by $X_{f}(\QQ)$:
\begin{align}
X_{f}(\QQ) := \left\{ x \in X(\QQ) \ \middle|\ \forall n \geq 0, f^{n}(x) \notin I_{f}  \right\}.
\end{align}
For a point $x \in X_{f}(\QQ)$, its $f$-orbit is 
\begin{align}
O_{f}(x) = \left\{ f^{n}(x) \ \middle|\ n \geq 0 \right\}.
\end{align}
\end{defn}

As $X_{f}(\QQ)$ is defined by countably many conditions,
it is not obvious that this set is non-empty.
It is, however, known that this set contains many points (\cite[Corollary 9]{Am11}, \cite[Proposition 3.24]{Xi22}).
For example, there is a non-empty adelic open subset $A \subset X(\QQ)$ (in the sense of Xie \cite{Xi22}) such that
$A \subset X_{f}(\QQ)$ and the $f$-orbit of every $x \in A$ is infinite.

A point $x \in X(\QQ)$ can be represented using a finite number of bits 
(for instance, by taking the coordinates and considering the minimal polynomial of each coordinate). 
In typical situations, the number of bits required to represent $f^{n}(x)$ increases as $n$ goes to infinity. 
The Kawaguchi-Silverman conjecture relates this growth rate with the geometric complexity of $f$, the dynamical degree. 
Specifically, it predicts that the arithmetic complexity of any dense orbit equals the geometric complexity of the map $f$. 
For more insight into the idea behind the conjecture, see \cref{sec:understand-ksc}.

The arithmetic complexity is measured by using Weil height functions.
While there are many ways to define height function,
we present the shortest one here. 
For more about the definitions and fundamentals on height functions, see 
\cite{HS00,BG06,La83}.

\begin{defn}[Weil height functions]\ 
\begin{enumerate}
\item
Let $K$ be a number field.
For a point $x = (a_{0}:\cdots : a_{N}) \in \P^{N}(\QQ)$ with $a_{i} \in K$, we define 
\begin{align}
h_{\P^{N}}(x) = \sum_{v \in M_{K}} \log \max\{ |a_{0}|_{v}, \dots, |a_{N}|_{v} \}
\end{align}
where $M_{K}$ is the set of absolute values on $K$ normalized as in \cite[p11, (1.6)]{BG06}.
That is, if $v$ is a place of $K$ whose restriction on $\Q$ is $p$ ($p$ is a prime number or $\infty$),
then
\begin{align}
|a|_{v} = |N_{K_{v}/\Q_{p}}(a)|_{p}^{1/[K:\Q]}
\end{align}
where $|\ |_{p}$ is the usual $p$-adic absolute value with $|p|_{p}=p^{-1}$ when $p$ is prime 
and $|\ |_{\infty}$ is the usual Euclidean absolute value on $\Q$.
The value $h_{\P^{N}}(x)$ is independent of the choice of $K$ and the homogeneous coordinates.
Thus $h_{\P^{N}} \colon \P^{N}(\QQ) \longrightarrow \R$ is a well-defined function, which is called 
the naive height function.

\item
Let $X$ be a projective variety over $\QQ$ and let $D$ be a Cartier divisor on $X$.
We attach $D$ a real valued function on $X(\QQ)$ that is well-defined up to difference by bounded functions.
Write $D = H_{1} - H_{2}$ with very ample divisors $H_{1}$ and $H_{2}$ on $X$.
Take closed immersions $\f_{i} \colon X \longrightarrow \P^{N_{i}}_{\QQ}\ (i=1,2)$
such that $\f_{i}^{*}\O_{\P^{N_{i}}}(1) \simeq \O_{X}(H_{i})$. 
Then we define 
\begin{align}
h_{D} = h_{\P^{N_{1}}} \circ \f_{1} -  h_{\P^{N_{2}}} \circ \f_{2}.
\end{align}
For another choice of $H_{i}$'s and $\f_{i}$'s, this function changes at most a bounded function.
Moreover, if $D$ and $D'$ are linearly equivalent, then $h_{D} - h_{D'}$ is also bounded on $X(\QQ)$.
This correspondence $D \mapsto h_{D}$ defines a group homomorphism
\begin{align}
&\{ \text{Cartier divisors on $X$} \}/\{ \text{principal divisors} \}  \\
&\longrightarrow {\mathrm{Map}}(X(\QQ), \R)/\{\text{bounded functions}\}.
\end{align}
For $\R$-Cartier divisors, we extend by linearity (jsut $ {\otimes}_{\Z} \R$ in the left hand side).
\end{enumerate}

\end{defn}

Now let us introduce the arithmetic degree.

\begin{defn}
Let $X$ be a quasi-projective variety over $\QQ$ and let $f \colon X \dashrightarrow X$ be a dominant rational map.
Let us fix an immersion $i \colon X \longrightarrow P$ to a projective variety $P$ over $\QQ$ and fix an ample divisor $H$ on $P$.
Let $h_{H} \colon P(\QQ) \longrightarrow \R$ be a logarithmic Weil height function associated with $H$. 
We set $h_{H}^{+} := \max\{1, h_{H}\}$.
For a point $x \in X_{f}(\QQ)$, the arithmetic degree of $x$ is
\begin{align}
\alpha_{f}(x) := \lim_{n \to \infty} h_{H}^{+}(f^{n}(x))^{1/n}
\end{align}
provided the limit exists.
As we mention below, we do not know if this limit exists always.
We use the following notation and call them upper and lower arithmetic degrees:
\begin{align}
&\overline{\alpha}_{f}(x) := \limsup_{n \to \infty}h_{H}^{+}(f^{n}(x))^{1/n}\\
& \underline{\alpha}_{f}(x) := \liminf_{n \to \infty} h_{H}^{+}(f^{n}(x))^{1/n}.
\end{align}
\end{defn}

\begin{rmk}
The definition apparently depends on the choice of immersion $i \colon X \longrightarrow P$,
the ample divisor $H$, and the height function $h_{H}$, but it is actually independent of all of them.
The independence of latter two is immediate from the general fact on height functions. 
That is, if $H$ and $H'$ are ample divisors on $P$ and $h_{H}$ and $h_{H'}$ are arbitrary choice of 
height functions associated with $H$ and $H'$ respectively, then
there are constants $C \geq 1$ and $C' \geq 0$ such that
\begin{align}
\frac{1}{C} h_{H} - C' \leq h_{H'} \leq C h_{H} + C'
\end{align}
on $P(\QQ)$.
It is easy to see from these inequalities that the definition is independent of the choice of $H$ and $h_{H}$.
The independence of the immersion $i \colon X \longrightarrow P$ follows from \cite[Lemma 3.8]{JSXZ21}.
\end{rmk}

We quickly recall the definition of dynamical degrees.

\begin{defn}[Dynamical degree]
Let $X$ be a projective variety over $\QQ$ and let $f \colon X \dashrightarrow X$ be a dominant rational map.
For each positive integer $n$, consider the following diagram
 
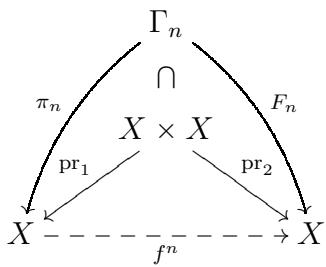
\[
\xymatrix{
&\G_{n} \ar@{}[d]|\bigcap \ar@/_10pt/[ldd]_{\pi_{n}} \ar@/^10pt/[rdd]^{F_{n}}&\\
&X \ar[ld]_{\pr_{1}} \times X \ar[rd]^{\pr_{2}} &\\
X \ar@{-->}[rr]_{f^{n}}&& X
}
\]
where $\G_{n} \subset X \times X$ is the graph of $f^{n}$. 
Let us fix an ample divisor $H$ on $X$.
For an integer $p$ between $0$ and $\dim X$, the $p$-th dynamical degree of $f$ is
\begin{align}
\d_{p}(f) := \lim_{n \to \infty} \left( F_{n}^{*}H^{p} \cdot \pi_{n}^{*}H^{\dim X -p}  \right)_{\G_{n}}^{1/n}.
\end{align}
It is known that this limit exists and independent of the choice of $H$ (cf.\ \cite{DS05,Da20,Tr20}).
Moreover, $\d_{p}(f)$ are invariant under birational conjugation.
We define dynamical degrees of dominant rational self-maps on quasi-projective varieties 
by taking a projective model.
\end{defn}

Note that if $f \colon X \longrightarrow X$ is a surjective self-morphism on a projective variety over $\QQ$,
then $\d_{1}(f)$ is equal to the spectral radius of $f^{*} \colon N^{1}(X)_{\R} \longrightarrow N^{1}(X)_{\R}$,
i.e.\ the maximum of the absolute value of the eigenvalues of $f^{*}$.

\subsection{Kawaguchi-Silverman conjecture}

Kawaguchi-Silverman conjecture is first stated for dominant rational self-maps on
projective spaces in \cite[Conjecture 1]{Sil12} and later formulated for arbitrary smooth projective varieties 
in \cite[Conjecture 6]{KS16b}.

\begin{conj}[Kawaguchi-Silverman Conjecture --original version-- {\cite[Conjecture 6]{KS16b}}]\label{KSC}
Let $X$ be a smooth projective variety and $f \colon X \dashrightarrow X$ a dominant rational map
both defined over $\QQ$.
\begin{enumerate}
\item
For all $x \in X_{f}(\QQ)$, the arithmetic degree $ \alpha_{f}(x)$ exists, i.e.\ the limit 
\begin{align}
\lim_{n \to \infty} h_{H}^{+}(f^{n}(x))^{1/n}
\end{align}
exists. 
\item
$ \alpha_{f}(x)$ is an algebraic integer.

\item
The collection of arithmetic degrees
\begin{align}
\left\{ \alpha_{f}(x) \ \middle|\ x \in X_{f}(\QQ) \right\}
\end{align}
is a finite set.

\item
If the orbit $O_{f}(x)$ is Zariski dense in $X$, then $ \alpha_{f}(x) = \d_{1}(f)$.

\end{enumerate}
\end{conj}

Understanding the growth rate of height functions along orbits is a fundamental subject in arithmetic dynamics. 
The arithmetic degree is the simplest measure of this growth rate, and the Kawaguchi-Silverman conjecture predicts that it is governed by the geometry of $f$. 
Given the foundational nature of this conjecture, we believe it has potential applications to other problems within arithmetic dynamics.
Indeed, there are already non-trivial applications of studying arithmetic degrees to the Zariski dense orbit conjecture 
(Zhang and Medvedev-Scanlon's conjecture) \cite{JSXZ21,MW22}. 
Additionally, it's worth mentioning that the study of the Kawaguchi-Silverman conjecture has stimulated 
deeper understanding of the structure of self-maps on varieties.

Let us first remark that part (2) and (3) of the conjecture 
are already disproven.
Lesieutre and Satriano constructed a counter example to part (3)
using a family of birational self-maps on surfaces whose dynamical degrees take infinitely many values.

\begin{thm}[{\cite[Theorem 2]{LS20}}]\label{thm:inf-set-of-ad}
Let $f \colon \P^{4}_{\QQ} \dashrightarrow \P^{4}_{\QQ}$
be the birational map defined by
\begin{align}
[X:Y:Z:A:B] \mapsto [XY+AX : YZ + BX : XZ : AX : BX ].
\end{align}
Then there exists a sequence of points $P_{n} \in (\P^{4}_{\QQ})_{f}(\QQ)$ for which 
$ \alpha_{f}(P_{n})$ exists and $\{ \alpha_{f}(P_{n})\}_{n}$ is an infinite set.
\end{thm}

Recently, Bell, Diller, and Jonsson constructed an example of 
dominant rational self-map whose first dynamical degree is a transcendental number \cite{BDJ20}
(cf.\ \cite{BDJK21} as well).
Wang and the author gave a counter example to part (2)
by showing that the map admits a point whose arithmetic degree is equal to the first dynamical degree.

\begin{thm}[{\cite[Theorem D]{MW22}}]
There is a dominant rational self-map $f \colon \P^{2}_{\QQ} \dashrightarrow \P^{2}_{\QQ}$
and $x \in (\P^{2}_{\QQ})_{f}(\QQ)$ such that $ \alpha_{f}(x)$ exists, $ \alpha_{f}(x) = \d_{1}(f)$, and
$ \alpha_{f}(x)$ is a transcendental number.
\end{thm}

The existence of the limit defining arithmetic degree (\cref{KSC}(1)) has been known
for morphisms (i.e.\ $I_{f} =  \emptyset$) since when the conjecture was formulated.
Moreover, when $f$ is a morphism, part (2)(3) are actually true.

\begin{thm}[{\cite[Theorem 3, Remark 23]{KS16a}}]\label{thm:ad-of-mor}
Let $X$ be a normal projective variety and let $f \colon X \longrightarrow X$ be a surjective self-morphism
both defined over $\QQ$.
\begin{enumerate}
\item
For any $x \in X(\QQ)$, $ \alpha_{f}(x)$ exists and is an algebraic integer.
\item
The set
\begin{align}
\left\{ \alpha_{f}(x) \ \middle|\ x \in X(\QQ) \right\}
\end{align}
is a finite set.
\end{enumerate}
More precisely, $ \alpha_{f}(x)$ is either $1$ or 
the absolute value of an eigenvalue of $f^{*} \colon N^{1}(X)_{\Q} \longrightarrow N^{1}(X)_{\Q}$.
\end{thm}

We already know that parts (2) and (3) of \cref{KSC} are not generally true for rational maps, 
and parts (1) - (3) have been proven for morphisms. 
Therefore, main interest lies in part (4), namely $ \alpha_{f}(x) = \d_{1}(f)$ for dense orbits. 
Nonetheless, it is also intriguing to investigate which maps satisfy parts (2) and (3), 
given that these parts still hold true for certain special rational maps like monomial maps. 
Further, the existence of the arithmetic degree (part (1) of \cref{KSC}) for rational maps remains a challenge. 
But for dense orbits, this part is typically proven during the process of proving $ \alpha_{f}(x) = \d_{1}(f)$."

In this note, when we refer to the ``Kawaguchi-Silverman conjecture", it primarily means this $ \alpha_{f} = \d_{1}(f)$ equality. 
While the Kawaguchi-Silverman conjecture is originally stated only for smooth projective varieties, there is no specific reason to exclude singular varieties.

To simplify the exposition, let us introduce the following conjecture, which is the essential part of the Kawaguchi-Silverman conjecture.
\begin{conj}[Kawaguchi-Silverman $\alpha=\delta$ conjecture]\label{a=dconj}
Let $X$ be a quasi-projective variety defined over $\QQ$.
Let $f \colon X \dashrightarrow X$ be a dominant rational map.
If $x \in X_{f}(\QQ)$ has Zariski dense orbit, then $ \alpha_{f}(x)$ exists and $ \alpha_{f}(x) = \d_{1}(f)$.
\end{conj}

\begin{rmk}[Smooth vs singular]
At this moment, we do not know if we can reduce the conjecture to the case where $X$ is smooth.
It is always possible to take a resolution of singularities of $X$, i.e.\ projective birational morphism
$\pi \colon X' \longrightarrow X$ with $X'$ is a smooth quasi-projective variety.
Since $\pi$ is birational, $f$ induces a dominant rational map $f' \colon X' \dashrightarrow X'$ 
such that $f \circ \pi = \pi \circ f'$.
But we do not know if all Zariski dense orbits of $f$ comes from that of $f'$, that is
for any point $x \in X_{f}(\QQ)$, it is not obvious if there is some $l \geq 0$ and $x' \in X_{f'}(\QQ)$
such that $f^{l}(x) = \pi(x')$. 
Moreover, even if this is the case, we do not know if the existence of arithmetic degrees
$ \alpha_{f}(x)$ and $ \alpha_{f'}(x')$ are equivalent. See \cref{prop:semconj-by-gene-fin} for a special case of such equivalence.
\end{rmk}

\begin{rmk}[Other ground fields]\label{rmk:other-fld}
The notion of height functions can be generalized to settings beyond number fields.
It is well known that the Weil height machine is available over fields with proper set of absolute values satisfying product formula.
One typical example of such fields are function fields. 
But if the coefficient field is infinite, we do not have the Northcott finiteness property and
the exact analogy of \cref{a=dconj} does not hold.
See \cite{MSS18b} and \cite[section 2]{Xi23} for more about this direction.
In this survey, we focus on number fields as this is the case that has been most extensively investigated.
\end{rmk}

Let us introduce the following fundamental inequality between arithmetic degree and dynamical degree.

\begin{thm}[{\cite{KS16b},\cite[Theorem 1.4]{Ma20a},\cite[Proposition 3.11]{JSXZ21}}]\label{thm:fund-ineq}
Let $X$ be a quasi-projective variety defined over $\QQ$.
Let $f \colon X \dashrightarrow X$ be a dominant rational map.
For $x \in X_{f}(\QQ)$, we have
\begin{align}
\overline{\alpha}_{f}(x) \leq \d_{1}(f).
\end{align}
\end{thm}

This theorem says that the maximal arithmetic complexity that the orbits can take
is at most the first dynamical degree.
By this theorem, to prove \cref{a=dconj}, it is enough to show the inequality 
$ \underline{\alpha}_{f}(x) \geq \d_{1}(f)$ for $x$ having Zariski dense orbit.

\subsection{Understanding Kawaguchi-Silverman conjecture}\label{sec:understand-ksc}

Roughly speaking, the height of a point is the number of bits you need to describe the point
(cf.\ \cite[Example 1.5.5, Proposition 1.6.6, Lemma 1.6.7]{BG06}).
Over rational numbers, the number of bits are just the number of digits of coordinates and
Kawaguchi-Silverman conjecture can be interpreted as
\begin{align}
\lim_{n \to \infty} \left( \text{number of digits of coordinates of $f^{n}(x)$} \right)^{1/n} = \d_{1}(f).
\end{align}
For a dominant rational self-map $f \colon \P^{N}_{\Q} \dashrightarrow \P^{N}_{\Q}$,
the dynamical degree $\d_{1}(f)$ has the following equivalent definition:
\begin{align}
\d_{1}(f) = \lim_{n \to \infty} \left( \text{algebraic degree of $f^{n}$}\right)^{1/n}
\end{align}
where algebraic degree of $f^{n}$ is the degree of the coprime homogeneous polynomials defining $f^{n}$.
If $f$ is defined by $N+1$ coprime homogeneous polynomials $P_{0},\dots ,P_{N}$ with integer coefficient of degree $d \geq 2$,
then the $n$-times composite of them $P_{0}^{(n)}, \dots, P_{N}^{(n)}$  defines $f^{n}$ and have degree $d^{n}$.
But they may have common factors, and dividing out all the common factors, we get the coprime defining polynomials
and their degree is the algebraic degree of $f^{n}$.
Now consider a rational point $x \in \P^{N}(\Q)$.
We can write $x = (a_{0}:\cdots :a_{N})$ with coprime integers $a_{0},\dots, a_{N}$.
In this case, it is easy to see that the height $h_{\O_{\P^{N}}(1)}$ can be chosen so that the height of $x$ is $\log \max\{|a_{0}|, \dots, |a_{N}|\}$,
which is approximately equal to the number of digits up to constant multiple.
The homogeneous coordinates of $f^{n}(x)$ are computed by applying the defining homogeneous polynomials of $f$ repeatedly:
$f^{n}(x) = (P_{0}^{(n)}(a_{0},\dots,a_{N}): \cdots:P_{N}^{(n)}(a_{0},\dots,a_{N}))$.
After $n$-times iterations, we might expect the number of digits to be $d^{n}$-times larger, but they may have common divisors.
Indeed, there will always be a common factor coming from the common factor of the polynomials $P_{0}^{(n)}, \dots, P_{N}^{(n)}$.
Kawaguchi-Silverman conjecture essentially claims that this is the only common factor:
\begin{align}
&\lim_{n \to \infty} 
\left( \log \max\left\{ \left| \frac{P_{0}^{(n)}(a_{0},\dots,a_{N})}{G_{n}} \right|,\dots,   \left| \frac{P_{N}^{(n)}(a_{0},\dots,a_{N})}{G_{n}} \right|   \right\} \right)^{1/n}\\
& = \lim_{n \to \infty} \left( \deg  \frac{P_{0}^{(n)}}{ \gcd(P_{0}^{(n)}, \dots, P_{N}^{(n)}) } \right)^{1/n}
\end{align}
where $G_{n} = \gcd\left(P_{0}^{(n)}(a_{0},\dots,a_{N}),\dots ,P_{N}^{(n)}(a_{0},\dots,a_{N}) \right)$.

To end this section,  we present an example of numerical computation.
Let us consider the birational self-map
\begin{align}
f \colon \P^{2}_{\Q} \dashrightarrow \P^{2}_{\Q} , \left(x:y:z \right) \mapsto \left(x(y+z) : x(y-z) : (y-z)^{2} \right).
\end{align}
We can easily show that the algebraic degrees $\deg_{a}f^{n}$ of $f^{n}$ satisfies
\begin{align}
&\deg_{a}f^{0} =1, \deg_{a} f=2\\
&\deg_{a}f^{n+2} \leq \deg_{a} f^{n+1} + \deg_{a} f^{n} \quad n \geq 0.
\end{align}
Thus the degree sequence is bounded above by Fibonacci sequence and
the first dynamical degree of $f$ is at most the golden ratio $(1 + \sqrt{5})/2 \approx 1.618..$.
The following is the first few iterates of the point $(5:3:1)$.
Here we write $f^{n}(5:3:1) = (a(n) : b(n) : c(n))$ where $a(n), b(n), c(n)$ are coprime integers.

\begin{center}
\renewcommand{\arraystretch}{1.1}
\begin{tabular}{l|lll}
$n$ & $a(n)$ & $b(n)$ & $c(n)$ \\[1mm] \hline
$0$ & $5$ & $3$ & $1$ \\
$1$ & $10$ & $5$ & $2$ \\
$2$ & $70$ & $30$ & $9$ \\
$3$ & $130$ & $70$ & $21$ \\
$4$ & $1690$ & $910$ & $343$ \\
$5$ & $302510$ & $136890$ & $45927$ \\
$6$ & $682765070$ & $339718730$ & $102151449$ \\
$7$ & $268649620388110$ & $144436902263290$ & $50256645593707$ \\
\end{tabular}
\end{center}

For this orbit, the sequence defining arithmetic degree
\begin{align}
h_{\O_{\P^{2}}(1)}(f^{n}(5:3:1))^{1/n} = \left(\log \max\{  |a(n)|, |b(n)|, |c(n)|  \} \right)^{1/n} \quad (n \geq 1)
\end{align}
is 
\begin{align}
2.302.., 2.061.., 1.694.., 1.651.., 1.660.., 1.652.., 1.649.., 1.645.., 1.643.. \cdots.
\end{align}

\section{A list of known results}

There are many special cases where the equality $ \alpha_{f}(x) = \d_{1}(f)$ is known.
The following is a rough summary of these results. 
This list may not be exhaustive and we do not observe the historical order.
Additionally, we do not provide definitions for the terminologies used here.
Let us restate \cref{a=dconj}.

\begin{conj*}[=\cref{a=dconj}. Kawaguchi-Silverman $\alpha=\delta$ conjecture]
Let $X$ be a quasi-projective variety defined over $\QQ$.
Let $f \colon X \dashrightarrow X$ be a dominant rational map.
If $x \in X_{f}(\QQ)$ has Zariski dense orbit, then $ \alpha_{f}(x)$ exists and $ \alpha_{f}(x) = \d_{1}(f)$.
\end{conj*}

This conjecture is proven in the following cases.

\begin{enumerate}
\item
$X$ is an arbitrary projective variety and $f$ is a polarized surjective self-morphism.
This case includes the case where $f$ is a surjective self-morphism on a curve, 
$\P^{n}_{\QQ}$, or more generally a projective variety with Picard number one.
(cf.\ \cref{thm:ksc-pol}.)

\item
$X$ is a projective surface and $f$ is a surjective self-morphism.
(cf.\ \cref{thm:ksc-surface,rmk:ksc-surf,thm:autos}.)

\item
$X$ is a semi-abelian variety and $f$ is a surjective self-morphism.
(cf.\ \cref{thm:ksc-monomial-maps,thm:ksc-semiabvar,rmk:ksc-abvar-all,thm:ksc-abvar}.)

\item
$X$ is a linear algebraic group and $f$ is a surjective group endomorphism.
(cf.\ \cref{thm:ksc-linalggrp}.)

\item
$X$ is a Mori dream space or (not necessarily $\Q$-factorial) variety of Fano type and $f$ is a surjective self-morphism.
This case includes the case where $X$ is a projective toric variety
and a smooth projective rationally connected variety admitting int-amplified endomorphism.
(cf.\ \cref{thm:sa-nef-cone,cor:mds}.)

\item
$X$ is a Hyper-K\"ahler variety and $f$ is a surjective self-morphism.
(cf.\ \cref{thm:autos}.)

\item
Automorphisms on $\A^{2}_{\QQ}$ and 
regular affine automorphisms on $\A^{n}_{\QQ}$.
(cf.\ \cref{thm:reg-affine-auto,rmk:ksc-A2,thm:endo-on-A2-JW}.)

\item
Self-morphisms on $\A^{2}_{\QQ}$ with some conditions on dynamical degrees.
(cf.\ \cref{thm:endo-on-A2-JW,rmk:JWX,thm:wang}.)

\item
There are many special cases proven by using minimal model theory and projective geometry.
These results are mostly on surjective self-morphisms on projective varieties.
Minimal model theory is particularly effective when the variety admits at least one polarized or
more generally int-amplified endomorphism.
See \cref{sec:mmp-proj-geom} for details.

\item
$X$ is a smooth projective variety, $f$ is a $1$-cohomologically hyperbolic rational map, 
and the $f$-orbit of $x$ is generic.
(cf.\ \cref{thm:cohhyp}.)

\end{enumerate}

\section{First properties}

Let us collect some of basic properties and observations 
that are frequently used in the study of Kawaguchi-Silverman conjecture.

\begin{prop}
Let $X$ be a quasi-projective variety defined over $\QQ$.
Let $f \colon X \dashrightarrow X$ be a dominant rational map.
Let $x \in X_{f}(\QQ)$.
For integers $m \geq 1$ and $l \geq 0$, we have
\begin{align}
&\underline{\alpha}_{f^{m}}(f^{l}(x)) = \underline{\alpha}_{f}(x)^{m}\\
&\overline{\alpha}_{f^{m}}(f^{l}(x)) = \overline{\alpha}_{f}(x)^{m}.
\end{align}
In particular, if one of $ \alpha_{f^{m}}(f^{l}(x))$ and $\alpha_{f}(x)$ exists,
then the other exists and $ \alpha_{f^{m}}(f^{l}(x)) = \alpha_{f}(x)^{m}$.
\end{prop}
\begin{proof}
Use the same argument as in \cite[Corollary 3.4]{Ma20a}.
The proof of \cite[Corollary 3.4]{Ma20a} uses \cite[Lemma 3.3]{Ma20a}, which assumes $X$
is smooth projective. Use \cite[Lemma 3.8]{JSXZ21} instead, then the same calculation works.
\end{proof}

The dynamical degree is also compatible with taking iterates: $\d_{1}(f^{m}) = \d_{1}(f)^{m}$.
Also the orbit $O_{f}(x)$ is Zariski dense if and only if
$O_{f^{m}}(f^{l}(x))$ is Zariski dense, so we can always replace $f$ with its iterates
to prove \cref{a=dconj}.

It is easy to see from the definition that the arithmetic degree is compatible with restriction to closed subvarieties.
That is, if $f \colon X \dashrightarrow X$ is a dominant rational map and 
$Y \subset X$ is a closed subvariety such that $Y \not\subset I_{f}$, $f(Y \setminus I_{f}) \subset Y$, and
$f|_{Y} \colon Y \dashrightarrow Y$ is dominant,
then we can form the following diagram:
 \[
\xymatrix{
Y \ar@{-->}[r]^{f|_{Y}} \ar@{}[d]|{\bigcap} & Y \ar@{}[d]|{\bigcap} \\
X \ar@{-->}[r]_{f} & X.
}
\]
If $x \in X_{f}(\QQ) \cap Y$, then $ \underline{\alpha}_{f}(x) = \underline{\alpha}_{f|_{Y}}(x)$
and $ \overline{\alpha}_{f}(x) = \overline{\alpha}_{f|_{Y}}(x)$.

When two maps are semi-conjugate by a dominant rational map,
we have inequalities between arithmetic degrees.

\begin{prop}\label{prop:ad-and-semiconj}
Let $X, Y$ be quasi-projective varieties defined over $\QQ$.
Consider the following commutative diagram of dominant rational maps:
 \[
\xymatrix{
X \ar@{-->}[r]^{f} \ar@{-->}[d]_{\pi} & X \ar@{-->}[d]^{\pi}\\
Y \ar@{-->}[r]_{g} & Y
}
\]
Let $x \in X_{f}(\QQ)$ be such that 
\begin{itemize}
\item the orbit $O_{f}(x)$ avoids the indeterminacy locus $I_{\pi}$ of $\pi$: $O_{f}(x) \cap I_{\pi} =  \emptyset$;
\item $\pi(x) \in Y_{g}(\QQ)$.
\end{itemize}
Then we have 
\begin{align}
&\underline{\alpha}_{g}(\pi(x)) \leq \underline{\alpha}_{f}(x)\\
& \overline{\alpha}_{g}(\pi(x)) \leq \overline{\alpha}_{f}(x).
\end{align}
\end{prop}
\begin{proof}
This follows from the definition of arithmetic degree and \cite[Lemma 3.8]{JSXZ21}.
\end{proof}

A typical application of this proposition is the following.
Suppose $f,g$, and $\pi$ are morphisms and $\d_{1}(f) = \d_{1}(g)$,
then for any point $x \in X(\QQ)$, we have $ \underline{\alpha}_{f}(x) \geq \underline{\alpha}_{g}(\pi(x))$.
If $O_{f}(x)$ is Zariski dense, then so is $O_{g}(\pi(x))$ and hence \cref{a=dconj} for $(X,f)$ follows from 
that of $(Y,g)$.
From this observation, we have for example:
\begin{lem}
Let $X, Y$ be projective varieties over $\QQ$ and $f \colon X \longrightarrow X$,
$g \colon Y \longrightarrow Y$ be surjective self-morphisms.
Then \cref{a=dconj} for $f \times g \colon X \times Y \longrightarrow X \times Y$
follows from that of $f$ and $g$.
\end{lem}
\begin{proof}
This follows from $ \alpha_{f \times g}(x,y) = \max\{ \alpha_{f}(x), \alpha_{g}(y)\}$
and $\d_{1}(f \times g) = \max\{ \d_{1}(f), \d_{1}(g)\}$.
\end{proof}
In general, there are self-morphisms on a product variety that are not of the form $f \times g$,
but there are several cases where this happens automatically (at least after some iterates).
See \cite{Sa20} for some results to this direction.

When two maps are conjugate by a generically finite morphism
and the self-maps are quasi-finite, we have the invariance of arithmetic degrees for dense orbits.

\begin{prop}[{\cite[Proposition 3.13]{JSXZ21}}]\label{prop:semconj-by-gene-fin}
Let $X, Y$ be quasi-projective varieties over $\QQ$ and consider the following commutative diagram
 \[
\xymatrix{
X \ar[r]^{f} \ar[d]_{\pi} & X \ar[d]^{\pi}\\
Y \ar[r]_{g} & Y
}
\]
where $f,g$ are quasi-finite dominant morphisms and $\pi$ is a generically finite dominant morphism.
Then for $x \in X(\QQ)$, the orbit $O_{f}(x)$ is Zariski dense in $X$ if and only if $O_{g}(\pi(x))$ is 
Zariski dense in $Y$.
Moreover, if $O_{f}(x)$ is Zariski dense in $X$, then 
$ \underline{\alpha}_{f}(x) = \underline{\alpha}_{g}(\pi(x))$ and $ \overline{\alpha}_{f}(x) = \overline{\alpha}_{g}(\pi(x))$.
In particular, \cref{a=dconj} for $f$ is equivalent to that of $g$.
\end{prop}

We conclude this section by remarking that Kawaguchi-Silverman conjecture 
is vacuous, i.e.\ there is no Zariski dense orbits, when the Kodaira dimension of the variety is positive.

\begin{prop}
Let $X$ be a quasi-projective variety over $\QQ$ and 
$f \colon X \dashrightarrow X$ be a dominant rational map.
If the Kodaira dimension of $X$ is positive, then $f$ does not have any Zariski dense orbits.
Here the Kodaira dimension of $X$ is by definition the Kodaira dimension of the smooth projective model of $X$.
In particular, \cref{a=dconj} is vacuous in this case.
\end{prop}
\begin{proof}
This is an application of a variant of the finiteness of pluricanonical representation.
See \cite[Theorem A]{NZ09}, \cite[Remark 1.2]{MSS18a}.
\end{proof}

As we have this proposition, the study of Kawaguchi-Silverman conjecture focuses on
varieties of non-positive Kodaira dimension.

\section{Canonical heights}\label{sec:canonical-height}

A typical strategy to prove the Kawaguchi-Silverman conjecture is to construct a height function that is compatible with the self-map.
Such a function is often referred to as a canonical height function. 
This function is a dynamical analogue of the N\'eron-Tate height function on abelian varieties.
Ever since Call and Silverman established the polarized endomorphism case in \cite{CS93}, 
many variants have appeared, reflecting the structure of the dynamical system under consideration.

Now, let us outline a typical argument used to prove the Kawaguchi-Silverman conjecture using canonical height functions.

Let $f \colon X \dashrightarrow X$ be a dominant rational self-map
and $h_{H}$ be a height function associated with an ample divisor $H$ on $X$.
Let $K$ be a number field over which $X$ and $f$ are both defined.
Suppose we have constructed a function $\hh \colon X_{f}(\QQ) \longrightarrow \R$
with the following properties:
\begin{align}
&\hh \leq C_{1} h_{H} + C_{2} \quad \text{for some constants $C_{1}, C_{2} > 0$;}\\
&\hh \circ f = \d_{1}(f) \hh ;\\
(*) \quad & \txt{The set $ \left\{ x \in X(L) \ \middle|\ \hh(x) \leq 0 \right\}$ is not Zariski dense in $X$ \\for any field extension 
$K \subset L \subset \QQ$ with $[L : K] < \infty$.}
\end{align}
Note that if $\hh$ is itself equal to, for example, an ample height function, then the first and the third conditions are
automatically satisfied.
Now,  for $x \in X_{f}(\QQ)$ we have
\begin{align}
C_{1}h_{H}(f^{n}(x)) + C_{2} \geq \hh(f^{n}(x)) = \d_{1}(f)^{n} \hh(x).
\end{align}
If $\hh(x) \leq 0$, then $\hh(f^{n}(x)) = \d_{1}(f)^{n}\hh(x) \leq 0$ as well.
If the point $x$ is defined over $L$, where $L$ is a finite extension of $K$,
then all $x, f(x), f^{2}(x), \dots$ are defined over $L$.
Thus if the orbit $O_{f}(x)$ is Zariski dense, we must have $\hh(x) >0$.
Thus we get 
\begin{align}
&\underline{\alpha}_{f}(x) =  \liminf_{n \to \infty} h_{H}^{+}(f^{n}(x))^{1/n} \geq \liminf_{n \to \infty} h_{H}(f^{n}(x))^{1/n} \\
&= \liminf_{n \to \infty} (C_{1} h_{H}(f^{n}(x)) + C_{2})^{1/n}
\geq \liminf_{n \to \infty} (\d_{1}(f)^{n}\hh(x))^{1/n} = \d_{1}(f).
\end{align}
Since we know $ \overline{\alpha}_{f}(x) \leq \d_{1}(f)$ (\cref{thm:fund-ineq}),
we get $ \alpha_{f}(x)$ exists and is equal to $\d_{1}(f)$.

Let us add some remark on the condition $(*)$.
It is a weakened version of Northcott finiteness property for height functions associated with ample divisors
(see, for example, \cite[Theorem 2.4.9]{BG06} for the Northcott property of ample height functions).
For $(*)$, we do not need ampleness and we have the following.
\begin{prop}[cf.\ {\cite[Proposition 3.5]{Ma20b},\cite[Proposition 5.1]{LM21}}]\label{prop:weak-Northcott}
Let $X$ be a geometrically irreducible normal projective variety over a number field  $K$.
Let $D$ be an $\R$-Cartier divisor on $X_{ \overline{K}}$ and $h_{D}$ be a height function associated with $D$.
If one of the following conditions is satisfied, then  the set
\begin{align}
\left\{ x \in X(L) \ \middle|\ K \subset L \subset \overline{K}, [L:K] \leq d, h_{D}(x) \leq B \right\}
\end{align}
is not Zariski dense in $X_{ \overline{K}}$ for all $B, d \geq 1$.
\begin{enumerate}
\item
$D$ is Cartier and $\k(X, D) > 0$, i.e.\ $\dim H^{0}(X, \O_{X}(mD)) \geq 2$ for some $m \in \Z_{>0}$.

\item
$X_{ \overline{K}}$ is $\Q$-factorial and $D$ is $\R$-linearly equivalent to two effective $\R$-divisors, i.e.\ 
$\exists D', D''$ such that $D', D''$ are effective $\R$-divisors on $X_{ \overline{K}}$,
$D' \neq D''$, and $D \sim_{\R} D' \sim_{\R} D''$.

\end{enumerate}
\end{prop}
If $\hh$ can be understood as a height function associated with an $\R$-Cartier divisor $D$ and one of the conditions 
in \cref{prop:weak-Northcott} is satisfied, then we can confirm $(*)$.

Now, let us introduce several cases where we can construct a nice canonical height functions 
and prove Kawaguchi-Silverman conjecture by using them.

\subsection{Polarized endomorphisms and nef canonical height}

Polarized endomorphism is probably the first and easiest case that Kawaguchi-Silverman conjecture was proven.
This case is a model case in the study of the conjecture.

\begin{defn}\label{def:pol}
Let $X$ be a projective variety over $\QQ$ and let $f \colon X \longrightarrow X$
be a surjective self-morphism.
Then $f$ is said to be polarized if there is an ample Cartier divisor $H$ such that
$f^{*}H \sim dH$ for some integer $d > 1$.
\end{defn}

Note that in this case we have $d = \d_{1}(f)$.

\begin{rmk}
It is known that the linear equivalence in the definition can be replaced with numerical equivalence.
More precisely, $f$ is polarized if and only if there is a rational number $d > 1$ and a big $\R$-Cartier divisor $D$
such that $f^{*}D \equiv dD$ \cite[Proposition 1.1]{MZ18}.
\end{rmk}

\begin{thm}[{\cite{CS93},\cite[Theorem 5]{KS16b}}]\label{thm:ksc-pol}
Notation as in \cref{def:pol}.
For any $x \in X(\QQ)$, the limit
\begin{align}
\hh_{H,f}(x) := \lim_{n \to \infty} \frac{h_{H}(f^{n}(x))}{d^{n}}
\end{align}
exists.
Moreover the function $\hh_{H,f}$ satisfies 
\begin{align}
&\hh_{H,f} = h_{H} + O(1)\\
&\hh_{H,f} \circ f = d \hh_{H,f}.
\end{align}
In particular, \cref{a=dconj} is true for $f$.
Moreover, $ \alpha_{f}(x) = d = d_{1}(f)$ if and only if $x$ has infinite $f$-orbit.
\end{thm}

Note that for any surjective self-morphism $f \colon X \longrightarrow X$
on a projective variety $X$, there is always a numerically non-zero nef $\R$-divisor $D$ such that
$f^{*}D \equiv \d_{1}(f)D$. 
(This is an application of Perron-Frobenius type theorem \cite{Bir67}. Note that the nef cone $\Nef(X)$ is strictly convex and 
$f^{*}$ preserves the nef cone.)
If $\d_{1}(f) > 1$, we can construct the similar canonical height function
\begin{align}
\hh_{D,f}(x) = \lim_{n \to \infty} \frac{\hh_{D}(f^{n}(x))}{\d_{1}(f)^{n}}
\end{align}
called {\it nef canonical height function}.
This function also satisfies similar properties 
\begin{align}\label{nefcanht}
&\hh_{D,f} = h_{D} + O\left( \sqrt{h_{H}^{+}} \right)\\
&\hh_{D,f} \circ f = \d_{1}(f) \hh_{D,f}
\end{align}
where $h_{H}$ is any ample height function.
Here the $O(1)$ is replaced with $O\left( \sqrt{h_{H}^{+}} \right)$
due to the fact that $f^{*}D \equiv \d_{1}(f)D$ is numerical equivalence and not the linear equivalence.
Note that if $\d_{1}(f) > 1$ and at least $X$ is smooth, it is actually possible to choose nef divisor $D$
so that $f^{*}D \sim_{\R} \d_{1}(f)D$ (cf.\ \cite[Remark 5.11]{Sa17}) and $\hh_{D,f} = h_{D} + O\left( 1\right)$.
But the problem is that we do not have the Northcott property in general for $h_{D}$ when $D$ is nef and not ample.
Let us also remark that a function $\hh_{D,f}$ 
satisfying \cref{nefcanht} is unique.

\subsection{Abelian varieties}
For group endomorphisms on abelian varieties, 
Kawaguchi and Silverman gave a description of the vanishing locus of nef canonical height
function and as a consequence proved \cref{a=dconj}.

\begin{thm}[{\cite{KS16a}}]\label{thm:ksc-abvar}
Let $A$ be an abelian variety over $\QQ$ and let $f \colon X \longrightarrow X$
be a surjective group endomorphism.
Suppose $\d_{1}(f) > 1$.
Let $D$ be a non-zero symmetric nef $\R$-divisor on $A$ such that $f^{*}D \equiv \d_{1}(f)D$.
(We can always make $D$ symmetric by replacing $D$ with $D + [-1]^{*}D$. Note that $f$ commutes with $[-1]$.)
Then for $x \in A(\QQ)$, we can define the canonical height function
\begin{align}
\hh_{D,f}(x) := \lim_{n \to \infty} \frac{h_{D}(f^{n}(x))}{\d_{1}(f)^{n}}.
\end{align}
Moreover, there is a proper abelian subvariety $B \subset A$ such that
\begin{align}
\left\{ x \in A(\QQ) \ \middle|\ \hh_{D,f}(x) = 0 \right\} = B(\QQ) + A(\QQ)_{ \rm tors}
\end{align}
where $A(\QQ)_{ \rm tors}$ is the set of torsion points.

Since torsion points defined over a fixed number field are finite,
the vanishing locus of $\hh_{D,f}$ is Zariski non-dense over a fixed number field.
In particular, \cref{a=dconj} is true for $f$.
\end{thm}

\begin{rmk}
In \cite{KS16a}, instead of $\hh_{D,f}$, they use 
\begin{align}
\hat{q}_{A,D} (x) = \lim_{n \to \infty} \frac{h_{D}(nx)}{n^{2}},
\end{align}
which is usual quadratic part of canonical height function on abelian varieties.
As we choose $D$ symmetric, we can check that this function satisfies \cref{nefcanht} and hence $\hh_{D,f} = \hat{q}_{A,D}$.
\end{rmk}

\begin{rmk}\label{rmk:ksc-abvar-all}
By a following work by Silverman \cite{Sil17},
\cref{a=dconj} is proven for surjective self-morphism on abelian varieties.
Note that every dominant rational self-map on an abelian variety is automatically a morphism.
This follows from more general fact that every rational map from a smooth variety to abelian variety is 
a morphism because there is no rational curves on an abelian variety (cf.\ \cite[Corollary 1.5]{KM98}).
Therefore, \cref{a=dconj} is completely solved for abelian varieties.
\end{rmk}

Based on this abelian variety case and 
a prior work by Silverman on the monomial maps on algebraic torus (see \cref{subsec:monomial-maps}),
Sano and the author proved \cref{a=dconj} for all surjective self-morphisms on semi-abelian varieties.

\begin{thm}[{\cite[Theorem 1.1]{MS20}}]\label{thm:ksc-semiabvar}
Let $X$ be a semi-abelian variety over $\QQ$ and let $f \colon X \longrightarrow X$ be a surjective self-morphism.
Then \cref{a=dconj} holds for $f$.
\end{thm}
In this case, other parts of \cref{KSC} are actually true, i.e.\ the arithmetic degrees are algebraic integers
and they take only finitely many values. See \cite{MS20} for details.

Having seen that the conjecture has been solved for self-morphisms on semi-abelian varieties, 
it is natural to ask about its validity for an arbitrary algebraic groups (they are automatically quasi-projective cf.\
 \cite[\href{https://stacks.math.columbia.edu/tag/0BF7}{Tag 0BF7}]{stacks-project}).
As for linear algebraic groups and their group endomorphisms, the answer is positive.
\begin{thm}[{\cite[Theorem 6.1]{Ma20b}}]\label{thm:ksc-linalggrp}
Let $X$ be a connected linear algebraic group over $\QQ$ and let $f \colon X \longrightarrow X$
be a surjective group endomorphism.
Then \cref{a=dconj} holds for $f$.
\end{thm}
Note that it is very difficult to remove the assumption that
$f$ is a group endomorphism, because for example 
the unipotent group consisting of upper triangular  matrices with 
diagonal entries $1$ is an affine space as varieties.
More generally, every linear algebraic group is 
birational to a projective space as varieties (cf.\ \cite[14.14]{Bo91}).

\subsection{Monomial maps}\label{subsec:monomial-maps}

A monomial map is a group endomorphism of an algebraic torus $\Gm^{N}$.
It is determined by an $N \times N$ matrix $A=(a_{ij})$ with integer coefficients. 
The map corresponds to $A$ is
\begin{align}
\f_{A} \colon &\Gm^{N} \longrightarrow \Gm^{N} \\
&(X_{1},\dots, X_{N}) \mapsto 
(X_{1}^{a_{11}}X_{2}^{a_{12}} \cdots X_{N}^{a_{1N}}, \dots, X_{1}^{a_{N1}}X_{2}^{a_{N2}} \cdots X_{N}^{a_{NN}}).
\end{align}
This map is dominant if and only if $\det A \neq 0$.
Note that the first dynamical degree of $\f_{A}$ is equal to the spectral radius of $A$.
Silverman introduced the following canonical height function:

\begin{defn}
Let $h \colon \Gm^{N}(\QQ) \longrightarrow \R$ be the naive logarithmic height function (defined by embedding $\Gm^{N} \subset \P^{N}$).
Let 
\begin{align}
l(A) = \left(\txt{the size of the maximal Jordan blocks of $A$\\ corresponding to eigenvalues of maximal absolute value} \right)- 1.
\end{align}
Then define
\begin{align}\label{eq:canonical-ht-monomial-map}
\hh_{\f_{A}}(x) = \limsup_{n \to \infty} \frac{h(\f_{A}^{n}(x))}{n^{l(A)}\d_{1}(f)^{n}}
\end{align}
for $x \in \Gm^{N}(\QQ)$.
\end{defn}

\begin{rmk}
The factor $n^{l(A)}$ in the definition of canonical height corresponds to the following fact on the growth rate
of degree sequence \cite[Theorem 24]{Sil12}.
Let $\deg_{a} (\f_{A}^{n})$ be the algebraic degree, i.e.\ degree of defining (coprime) homogeneous polynomials of 
$\f_{A}$ considered as a rational self-map of $\P^{N}$.
Then $\deg_{a}(\f_{A}^{n}) \gg \ll n^{l(A)} \d_{1}(\f_{A})^{n}$ for $n \geq 1$.
\end{rmk}

\begin{thm}[{\cite[Theorem 4]{Sil12}}]\label{thm:ksc-monomial-maps}
Notation as above. Suppose $\d_{1}(\f_{A}) > 1$.
\begin{enumerate}
    \item The arithmetic degrees of $\f_A$ satisfy
    \[\{\alpha_{\f_A}(x) \mid x \in \Gm^{N}(\QQ)\} \subset \{\text{eigenvalues of } A\} \cup \{1\}.\]
    In particular, $\alpha_{\f_A}(x)$ is an algebraic integer for all $x \in \Gm^{N}(\QQ)$.
    
    \item Let $x \in \Gm^{N}(\QQ)$ be a point with $\hh_{\f_A}(x) = 0$.
     Then the orbit $O_{\phi}(x)$ is contained in a proper (possibly disconnected) algebraic subgroup of $\Gm^{N}$.
     In particular, \cref{a=dconj} holds for $\f_{A}$.
    
    \item If the matrix $A$ is diagonalizable over $\mathbb{C}$, then
    \[\hh_{\f_{A}}(x) = 0 \Longleftrightarrow \alpha_{\f_{A}}(x) < \delta_{1}(\f_{A}).\]
    
    \item If the characteristic polynomial of the matrix $A$ is irreducible over $\mathbb{Q}$, then
    \[\hh_{\f_{A}}(x) = 0 \Longleftrightarrow O_{\f_A}(x) \text{ is finite.}\]
\end{enumerate}

\end{thm}

\begin{rmk}
Based on this result, Lin proved the existence of arithmetic degrees for monomial maps on projective toric varieties.
(Note that \cref{a=dconj} for this case follows directly from \cref{thm:ksc-monomial-maps}.)
Moreover, Lin constructed an example of monomial maps for which the sequence in \cref{eq:canonical-ht-monomial-map}
has infinitely many accumulation points \cite[Example 3.1]{Li19}.
\end{rmk}

\subsection{Automorphisms on surfaces and Hyper-K\"ahler varieties}\label{auto-surf-hk}

When you have more than one eigendivisors and can construct canonical height functions 
associated with them, it is reasonable to expect that the sum of them might have desired positivity
even if each individual function does not.
This idea already appeared in \cite{Sil91}, and now has numerous variants.

Let us consider an automorphism $f \colon X \longrightarrow X$ with $\d_{1}(f)>1$.
Then we have two nef canonical height functions $\hh_{D_{+},f}$ and $\hh_{D_{-}, f^{-1}}$
with respect to $f$ and its inverse $f^{-1}$.
Here $D_{+}$ and $D_{-}$ are non-zero nef $\R$-divisors such that $f^{*}D_{+} \equiv \d_{1}(f) D_{+}$
and $(f^{-1})^{*}D_{-} \equiv \d_{1}(f^{-1})D_{-}$.
(Note that we have $\d_{1}(f^{-1})=\d_{\dim X -1}(f)$ and log concavity implies $\d_{\dim X -1}(f) > 1$.)
Then 
\begin{align}\label{auto-cano-ht}
\hh_{D_{+},f} + \hh_{D_{-}, f^{-1}} = h_{D_{+} + D_{-}} + O\left( \sqrt{h_{H}^{+}} \right)
\end{align}
would have positivity property if so does the divisor $D_{+} + D_{-}$.
Kawaguchi originally used this strategy for surfaces and later
Lesieutre and Satriano extended this idea to Hyper-K\"ahler varieties.

\begin{thm}[{\cite{Ka08}, \cite[Theorem 2]{KS14}, \cite[Theorem 1.2]{LS21}}]\label{thm:autos}
\cref{a=dconj} is true for automorphisms on 
\begin{enumerate}
\item smooth projective surfaces;
\item Hyper-K\"ahler varieties.
\end{enumerate}
In particular,  \cref{a=dconj} holds for all surjective self-morphisms on Hyper-K\"ahler varieties 
because a surjective self-morphism on a Hyper-K\"ahler variety is automatically an automorphism.
\end{thm}

\begin{rmk}\label{rmk:ksc-surf}
\cref{a=dconj} is now completely proven for surjective self-morphisms on projective surfaces (cf.\ \cref{thm:ksc-surface}).
Non-isomorphic surjective self-morphisms on smooth projective surfaces were proven by Sano, Shibata, and the author \cite{MSS18a},
and later Meng and Zhang proved the singular case \cite{MZ22}.
Except the Kawaguchi's automorphism case, the proof heavily relies on classification (or minimal model theory)
of surfaces.
\end{rmk}

\cref{thm:autos} follows from the following.
\begin{prop}[{\cite{Ka08},\cite{LS21}}]
Let $X$ be one of the classes of varieties in \cref{thm:autos}.
Let $f \colon X \longrightarrow X $ be an automorphism and suppose $\d_{1}(f) > 1$.
Then we have $\d_{1}(f^{-1}) = \d_{1}(f)$ and
there are nef $\R$-divisors $D_{+}$ and $D_{-}$ on $X$ such that
\begin{align}
&f^{*}D_{+} \sim_{\R} \d_{1}(f) D_{+}\\
&(f^{-1})^{*}D_{-} \sim_{\R} \d_{1}(f) D_{-}\\
&\text{$D_{+} + D_{-}$ is nef and big}.
\end{align}
In particular, the function $\hh_{D_{+},f} + \hh_{D_{-},f^{-1}} $ is not bounded above on any Zariski dense set of 
$K$-rational points with bounded extension degree $[K:\Q]$.
\end{prop}

Using the same method, Lesieutre and Satriano proved \cref{a=dconj} for
automorphisms on smooth projective varieties of Picard number two \cite[Theorem 2.30]{LS21}.

\begin{rmk}
Once you get the bigness of $D_{+} + D_{-}$, it is easy to deduce \cref{a=dconj}
from \cref{auto-cano-ht}.
Indeed, bigness of $D_{+} + D_{-}$ implies $h_{D_{+} + D_{-}} \geq C h_{H}^{+}$
on a non-empty open subset of $X$ where $C$ is a positive constant.
Therefore, for a point $x \in X(\QQ)$ with Zariski dense orbit, there is a sequence $n_{1}, n_{2},\dots$ such that
\begin{align}
 &h_{D_{+} + D_{-}} (f^{n_{j}}(x))+ O\left( \sqrt{h_{H}^{+}(f^{n_{j}}(x))} \right) \\
 &\geq h_{H}^{+}(f^{n_{j}}(x)) \left( C + O\left( \frac{1}{ \sqrt{h_{H}^{+}(f^{n_{j}}(x))}}\right)  \right)
\end{align}
and the right hand side goes to $\infty$ as $j \to \infty$.
Thus by \cref{auto-cano-ht}, 
\begin{align}
&\hh_{D_{+}, f}(f^{n_{j}}(x)) + \hh_{D_{-}, f^{-1}}(f^{n_{j}}(x)) \\
&= \d_{1}(f)^{n_{j}}\hh_{D_{+},f}(x) + \d_{1}(f^{-1})^{-n_{j}} \hh_{D_{-},f^{-1}}(x)
\end{align}
goes to infinity as $j \to \infty$.
Hence we must have $\hh_{D_{+},f}(x) > 0$ and this implies $ \alpha_{f}(x) \geq \d_{1}(f)$.
In particular, it is not necessary to find $D_{+}$ and $D_{-}$ so that
they are eigenvectors up to linear equivalence (numerical equivalence is enough).
\end{rmk}

\begin{rmk}
Jonsson and Reschke extended the construction of $\hh_{D_{+}}$ and $\hh_{D_{-}}$ to 
birational self-maps on smooth projective surfaces \cite[Theorem D]{JR18}.
But in the paper they say they were not able to prove that their canonical height is not constant zero.
By \cref{thm:cohhyp,rmk:cohhyp-genorb}, Kawaguchi-Silverman conjecture is now almost proven in this case
and we know that there are points with arithmetic degree equal to dynamical degree,
but we still do not know wether Jonsson-Reschke's canonical height is not identically zero or not (cf.\ \cite[Remark 4.11]{MW22}).
\end{rmk}

\subsection{Mori dream spaces and varieties of Fano type}

A surjective self-morphism $f \colon X \longrightarrow X$ on a normal projective variety over $\QQ$
is automatically a finite morphism and it induces a linear automorphism $f^{*}$ on $N^{1}(X)_{\R}$ by pull-back.
Moreover, $f^{*}$ induces a bijection on the nef cone $\Nef(X)$.
Therefore, if $\Nef(X)$ is generated by finitely many vectors, it is generated by finitely many extremal rays and
these extremal rays are permuted by $f^{*}$.
By replacing $f$ with $f^{n}$ for some $n \geq 1$, we may assume all the extremal rays are eigenvectors of $f^{*}$.
In particular, one of them is the leading eigenvector, i.e.\ eigenvector with eigenvalue $\d_{1}(f)$.
If we further assume all the extremal rays of $\Nef(X)$ are generated by base point free Cartier divisors, 
then we can find a nef canonical height function with nice properties.

\begin{thm}[{\cite{Ma20b}}]\label{thm:sa-nef-cone}
Let $f \colon X \longrightarrow X$ be a surjective self-morphism on a normal projective variety over a number field $K$.
Suppose 
\begin{enumerate}
\item $N^{1}(X_{\QQ})_{\Q} \simeq \Pic(X_{\QQ})_{\Q}$
\item the nef cone $\Nef(X_{\QQ})$ is generated by finitely many base point free Cartier divisors
\end{enumerate}
Then after replacing $f$ with its iterates, there is a non-zero nef and base point free divisor $D$ such that
$f^{*}D \sim \d_{1}(f)D$ and the corresponding nef canonical height function $\hh_{D,f}$ satisfies
the following property.
For any $B \in \R$ and $d \geq 1$,
\begin{align}
\left\{ x \in X(L) \ \middle|\ [L:K] \leq d, \hh_{D,f}(x) \leq B \right\}
\end{align}
is not Zariski dense.
In particular, \cref{a=dconj} is true for $f$.
\end{thm}

The assumptions in the theorem are well-studied in algebraic geometry
and there are many varieties that satisfy them.

\begin{cor}\label{cor:mds}
\cref{a=dconj} is true for surjective self-morphism on
Mori dream spaces, varieties of Fano type, and projective toric varieties.
\end{cor}
\begin{proof}
The conditions in \cref{thm:sa-nef-cone} are actually parts of the definition of Mori dream spaces.
We can check that varieties of Fano type satisfy the conditions by basic theorems in minimal model theory
(base point free theorem and cone theorem).
For projective toric varieties, the conditions follow from
\cite[Theorem 6.3.12, Theorem 6.3.15, Theorem 6.3.20]{CLS}.
\end{proof}

The key point in \cref{thm:sa-nef-cone}, as well as polarized case, is that we can find 
a good eigendivisor for $f^{*}$.
In \cite{Na20}, Nasserden explored what conditions on eigenspaces of $f^{*}$ would be sufficient for 
proving Kawaguchi-Silverman conjecture.

\subsection{Regular affine automorphisms}\label{subsec:regular-affine-auto}

When our self-map is a rational map and not a morphism,
it is usually very difficult to construct canonical height function with nice properties.
There is an early work by Kawaguchi on regular affine automorphisms that can be understood
as a generalization of the strategy in \cref{auto-surf-hk}.

\begin{defn}
An automorphism $f \colon \A^{N} \longrightarrow \A^{N}$ over $\QQ$
is called regular affine automorphism (regular polynomial automorphism) if
the indeterminacy loci of $f$ and $f^{-1}$ considered as birational maps on $\P^{N}$ are disjoint.
\end{defn}

Regular affine automorphism is algebraically stable when it is considered as a rational self-map 
on the projective space $\P^{N}$ via the natural embedding $\A^{N} \subset \P^{N}$
(i.e.\ $(f^{n})^{*} = (f^{*})^{n} \colon \Pic \P^{N} \longrightarrow \Pic \P^{N}$, cf.\ for example \cite[Lemma 19]{KS14}).

\begin{thm}[{\cite{Ka13}}]\label{thm:reg-affine-auto}
Let $f \colon \A^{N}_{\QQ} \longrightarrow \A^{N}_{\QQ}$ be a regular affine automorphism 
with degree $d \geq 2$ (i.e.\ $d$ is the max of the degrees of the coordinate polynomials).
Then $d = \d_{1}(f)$ since $f$ is algebraically stable as a self-map on $\P^{N}_{\QQ}$.
Let $d_{-}$ be the degree of $f^{-1}$ (note that $d_{-} = \d_{1}(f^{-1}) = \d_{N-1}(f) \geq 2$).
Let $h \colon \A^{N}(\QQ) \longrightarrow \R$ be the naive logarithmic height function (defined by embedding $\A^{N}_{\QQ} \subset \P^{N}_{\QQ}$).
Then the limits
\begin{align}
&\hh_{f}^{+} (x) = \lim_{\n \to \infty} \frac{h(f^{n}(x))}{d^{n}} \ \ x \in \A^{N}(\QQ)\\
&\hh_{f}^{-} (x) = \lim_{\n \to \infty} \frac{h(f^{-n}(x))}{d_{-}^{n}} \ \ x \in \A^{N}(\QQ)
\end{align}
exists and $\hh_{f}^{+} + \hh_{f}^{-} \gg \ll h + O(1)$ (implicit constants depend at most on $f$).
In particular, \cref{a=dconj} is true for $f$.
\end{thm}

\begin{rmk}\label{rmk:ksc-A2}
In \cite{Ka06}, Kawaguchi proved two dimensional version of \cref{thm:reg-affine-auto} for
all automorphisms on $\A^{2}_{\QQ}$ with dynamical degree $>1$.
We remark that, every automorphism $f \colon \A^{2}_{\QQ} \longrightarrow \A^{2}_{\QQ}$
with $\d_{1}(f) > 1$ is conjugate to a regular affine automorphism (more precisely, composite of H\'enon maps) cf.\ \cite[Theorem 3.1, Proposition 3.2]{Ka06}.
Here conjugate means conjugate by an automorphism on $\A^{2}_{\QQ}$.
Therefore, \cref{thm:reg-affine-auto} implies \cref{a=dconj} for all automorphisms on $\A^{2}_{\QQ}$.
\end{rmk}

\subsection{Self-morphisms on $\A^{2}_{\QQ}$}

Jonsson and Wulcan studied the case for self-morphisms on affine place $\A^{2}_{\QQ}$.
They use some results about equivariant compactification of self-morphisms on $\A^{2}_{\QQ}$,
which is based on prior works by Favre and Jonsson on valuative trees.

\begin{thm}[{\cite[Main Theorem]{JW12}}]\label{thm:endo-on-A2-JW}
Let $f \colon \A^{2}_{\QQ} \longrightarrow \A^{2}_{\QQ}$ be a dominant self-morphism
with $\d_{1}(f) > \d_{2}(f)$ ($f$ is said to have small topological degree).
Let $h \colon \A^{2}(\QQ) \longrightarrow \R$ be the naive logarithmic height function (defined by embedding $\A^{2} \subset \P^{2}$).
Then the limit 
\begin{align}
\hh(x) = \lim_{n \to \infty} \frac{h(f^{n}(x))}{\d_{1}(f)^{n}}
\end{align}
exists for all $x \in \A^{2}(\QQ)$.
Further, $\hh$ is not constant $0$ and $\hh(x) = 0$ implies 
\begin{align}
\overline{\alpha}_{f}(x) = \limsup_{n \to \infty} h(f^{n}(x))^{1/n} \leq \d_{2}(f) < \d_{1}(f).
\end{align}
Moreover, if $f$ is an automorphism, then $\hh(x) = 0$ if and only if $x$ is $f$-periodic.
In particular, \cref{a=dconj} holds for automorphisms on $\A^{2}_{\QQ}$.
\end{thm}

We remark that \cref{a=dconj} for automorphisms on $\A^{2}_{\QQ}$ was already proven by Kawaguchi in \cite{Ka06} (cf.\ \cref{rmk:ksc-A2}).
But Jonsson and Wulcan's proof is different and has the following extension.

\begin{rmk}\label{rmk:JWX}
There is an unpublished work by Jonsson-Wulcan-Xie on Kawaguchi-Silverman conjecture for self-morphisms on $\A^{2}$.
For a dominant self-morphism $f \colon \A^{2}_{\QQ} \longrightarrow \A^{2}_{\QQ}$, if $f$ satisfies one of the following conditions,
then \cref{a=dconj} holds for $f$:
\begin{enumerate}
\item $\d_{1}(f) \geq \d_{2}(f)$;
\item $\d_{1}(f) \notin \Z$;
\item $\d_{1}(f) \in \Z$ and $\gcd(\d_{1}(f), \d_{2}(f)) \neq 1$.
\end{enumerate}
Also, they proved that Vojta's conjecture implies Kawaguchi-Silverman conjecture for $f$.
I appreciate them sharing this result with me and allowing me to include it here.
\end{rmk}

Let us mention that there is another relevant result \cref{thm:wang,rmk:jwx-vs-wang},
which can be understood as a generalization of Kawaguchi's approach.

\section{Minimal Model Program and Projective geometry}\label{sec:mmp-proj-geom}

Recently, there have been many advancements in the study of the Kawaguchi-Silverman conjecture 
that use techniques from projective geometry and the minimal model program. 
The fundamental idea is as follows: 
Suppose we have a surjective self-morphism $f \colon X \longrightarrow X$ on a certain projective variety $X$. 
If we can find another projective variety $Y$ that is birational to $X$ 
and is geometrically well understood (e.g.\ $\P^{N}$, an abelian variety, or a product of them), 
then we could deduce the Kawaguchi-Silverman conjecture for $f$ from that of the induced map $f_{Y}$ on $Y$.
However, the challenge lies in the fact that even if $Y$ is a well-understood variety, 
the Kawaguchi-Silverman conjecture for $Y$ is very difficult if $f_{Y}$ is a rational map. 
Therefore, the crucial aspect of this strategy is to construct $Y$ in such a way that $f_{Y}$ is a morphism.

\subsection{Varieties admitting an int-amplified endomorphism}

For a projective variety, the presence of certain special types of surjective self-morphisms, 
such as polarized self-morphisms, generally imposes a very strong condition. 
There have been efforts to classify all such varieties and to understand the structure 
of these self-morphisms. 
In this context, Meng introduced the following class of surjective self-morphisms on projective varieties:

\begin{defn}[\cite{Me20}]
Let $f \colon X \longrightarrow X$ be a surjective self-morphism on a projective variety defined over an algebraically closed field.
Then $f$ is said to be int-amplified if there is an ample Cartier divisor $H$ such that
$f^{*}H - H$ is ample.
\end{defn}

\begin{rmk}
Polarized endomorphisms are int-amplified. 
Int-amplifiedness is characterized by the following property:
the modulus of all the eigenvalues of $f^{*} \colon N^{1}(X) \longrightarrow N^{1}(X)$ is greater than $1$ \cite[Theorem 3.3]{Me20}. 
Contrast to polarized case, being int-amplified is preserved under taking products. 
For additional properties of int-amplified endomorphisms, see Section 3 in \cite{Me20}.
\end{rmk}

Int-amplified endomorphism was first introduced in the study of equivariant minimal model program
as described below, but the notion seems to be useful in other context, e.g.\ \cite{KT23}.
Also int-amplifiedness can be understood as cohomological hyperbolicity (cf.\ \cref{def:coh-hyp}):
$f \colon X \longrightarrow X$ is int-amplified if and only if $\d_{\dim X -1}(f) < \d_{\dim X}(f) = \deg f$
(at least when $X$ is smooth), i.e.\ $f$ is $\dim X$-cohomologically hyperbolic.


Minimal model theory provides us strong tools to study the structure of algebraic varieties.
The basic idea it to perform a series of special type of transformations that simplify the variety.
Meng and Zhang proved that if the original variety admits an int-amplified endomorphism,
then all the process in MMP can be made equivariant under any surjective self-morphisms up to iterates.
The following is a part of the consequence of their results (\cite[Theorem 1.2]{MZ20}, \cite[Theorem 1.8]{MZ18},\cite[Theorem 1.10]{Me20}).

\begin{thm}[Equivariant MMP (Meng-Zhang)]\label{equiMMP}
Let $X$ be a $\Q$-factorial klt projective variety over $\QQ$ admitting an int-amplified endomorphism.
Then for any surjective self-morphism $f \colon X \longrightarrow X$, there exists a positive integer $n>0$ and 
the following commutative diagram of rational maps
\[
\xymatrix{
*+[l]{X= X_{0}}  \ar@{-->}[r]^{\pi_{0}} \ar[d]_{g_{0}=f^{n}} & X_{1} \ar[d]_{g_{1}}  \ar@{-->}[r]^{\pi_{1}} & \cdots \ar@{-->}[r]  
&  X_{r-1} \ar@{-->}[r]^{\pi_{r-1}} \ar[d]_{g_{r-1}} & X_{r} \ar[d]^{g_{r}}  & \ar[l]_{\nu} A \ar[d]^{h}\\
 X_{0} \ar@{-->}[r]_{\pi_{0}} & X_{1} \ar@{-->}[r]_{\pi_{1}}  & \cdots \ar@{-->}[r] & X_{r-1} \ar@{-->}[r]_{\pi_{r-1}} & X_{r}  & \ar[l]^{\nu} A 
}
\]
such that
\begin{enumerate}
\item $\pi_{i} \colon X_{i} \dashrightarrow X_{i+1}$ is either the divisorial contraction, flip, or Fano contraction of a 
$K_{X_{i}}$-negative extremal ray;

\item $A$ is an abelian variety  and $\nu$ is a finite morphism that is \'etale in codimension one, i.e.\ quasi-\'etale morphism (Note that $A$ could be a point);
\item $g_{i} \colon X_{i} \longrightarrow X_{i}$ for $i=0, \dots ,r$ and $h \colon A \longrightarrow A$ are surjective self-morphisms.

\end{enumerate}
\end{thm}

In \cite{MZ22}, Meng and Zhang studied a strategy of proving Kawaguchi-Silverman conjecture and
extracted the following special case called Case TIR as the obstruction to their strategy.

\begin{defn}[{\bf Case ${\rm TIR}_{n}$ --Totally Invariant Ramification case--}]\label{CaseTIR}
Let $X$ be a normal projective variety of dimension $n \geq 1$, 
which has only $\Q$-factorial klt singularities and admits an int-amplified endomorphism. 
Let $f : X \longrightarrow X$ be an arbitrary surjective self-morphism. Moreover, we impose the following conditions.

\begin{enumerate}
\item[(A1)] 
The anti-Kodaira dimension is zero: $\kappa(X, -K_{X} ) = 0$, and
$-K_{X}$ is nef, whose class is extremal in both the nef cone $\Nef(X)$ and the pseudo-effective divisors cone $ \overline{\Eff}(X)$.

\item[(A2)] There is a prime divisor $D$ such that $f^*D = \d_1(f) D$ and $D \sim_{\Q}-K_{X}$.
(Here the equality is an equality as divisors. In particular, $\d_{1}(f)$ is an integer.)

\item[(A3)] The ramification divisor of $f$ satisfies $\Supp R_f = D$.

\item[(A4)] There is an $f$-equivariant Fano contraction $\pi : X \longrightarrow Y$ with $\d_{1}(f) > \d_{1}(f_{Y}) (\geq 1)$,
where $f_{Y}$ is the induced surjective self-morphism on $Y$.

\item[(A5)] $\dim(X) \geq \dim(Y) + 2 \geq 3$.
\end{enumerate}
\end{defn}

\begin{thm}[{\cite[Theorem 1.7]{MZ22}}]
Let $X$ be a $\Q$-factorial klt projective variety over $\QQ$ admitting an int-amplified endomorphism.
\begin{enumerate}
\item
If $K_{X}$ is pseudo-effective, then \cref{a=dconj} is true for all surjective self-morphisms on $X$.
Note that in this case $X$ is an Q-abelian variety (i.e.\ there is a finite surjective morphism from an abelian variety
that is \'etale in codimension one).

\item
Assume \cref{a=dconj} is true for Case ${\rm TIR}_{n}$ for $n \leq \dim X$ (it is enough to assume for the fibrations appearing in an equivariant MMP starting from X).
Then \cref{a=dconj} is true for all surjective self-morphisms on $X$.

\end{enumerate}
\end{thm}

Note that when $X$ is a Q-abelian variety, then every surjective self-morphism on $X$ lifts to
a surjective self-morphism on the abelian variety covering it and hence \cref{a=dconj} follows from the abelian variety case
\cite[Theorem 2.8]{MZ22}.

The following theorem gives us some insight about the conditions in Case TIR.

\begin{thm}[{\cite[Proposition 1.6]{MZ22}}]
Let $f : X \longrightarrow X$ be a surjective self-morphism of a $\Q$-Gorenstein normal projective variety $X$ over $\QQ$
with positive anti-Kodaira dimension: $\kappa(X, -K_X) > 0$. Suppose $f^*K_X \equiv \d_{1}(f) K_X$. Then \cref{a=dconj} holds for $f$.
\end{thm}

\begin{rmk}
Using equivariant MMP, we may reduce the problem to the case where $f^*K_X \equiv \d_{1}(f) K_X$ holds.
This theorem ensures once you have $\kappa(X, -K_X) > 0$, then we get the desired result.
Further investigation of the case $\kappa(X, -K_X) = 0$ leads to the Case TIR (\cref{CaseTIR}).
See \cite{MZ22} for more details.
\end{rmk}

Yoshikawa and the author proved that if the original variety $X$
is smooth rationally connected, then Case TIR does not happen.

\begin{thm}[\cite{MY22}]
Let $X$ be a smooth projective rationally connected variety defined over $\QQ$ such that
$X$ admits at least one int-amplified endomorphism.
Then for any surjective self-morphism $f \colon X \longrightarrow X$, \cref{a=dconj} is true.
\end{thm}

\begin{rmk}
This theorem is first proven in dimension three by Meng and Zhang \cite[Theorem 1.11]{MZ22}.
After we proved this theorem, 
Yoshikawa proved that such a variety is actually a variety of Fano type \cite{Yo21}. 
Therefore Kawaguchi-Silverman conjecture follows from \cref{thm:sa-nef-cone}.
But let us remark that Yoshikawa's proof in \cite{Yo21} is an extension of our arguments in \cite{MY22}, 
so it does not give an easier proof of Kawaguchi-Silverman conjecture.
\end{rmk}

As another application of minimal model program,
Meng and Zhang proved Kawaguchi-Silverman conjecture for surjective self-morphisms on
projective surfaces (not necessarily smooth).

\begin{thm}[{\cite[Theorem 1.3]{MZ22}}]\label{thm:ksc-surface}
\cref{a=dconj} is true for all surjective self-morphisms on projective surfaces.
\end{thm}

\begin{rmk}
For automorphisms, this theorem follows from smooth case, which was proven by Kawaguchi \cite{Ka08},
by taking equivariant resolution of singularities.
For surjective self-morphisms on smooth projective surfaces, this was proven in \cite{MSS18a}.
\end{rmk}

\subsection{Projective geometry}

There are many works that utilize techniques from projective and birational geometry 
to prove special cases of Kawaguchi-Silverman conjecture.
A common strategy involves reducing the problem to cases that have already been proven, 
or demonstrating that there are no Zariski dense orbits, thereby confirming \cref{a=dconj} in a vacuous sense. 
This approach often requires exploiting the algebro-geometric structures of the varieties in question.
In this subsection, we collect such results.

A smooth projective variety $X$ over $\QQ$ is said to be Calabi--Yau if $\dim X \geq 3$,
$\O_{X}(K_{X}) \simeq \O_{X}$, and $\dim H^{0}(\Omega_{X}^{p}) = 0$ for $ 0 < p < \dim X$.
According to the Beauville-Bogomolov decomposition, 
Calabi-Yau varieties together with abelian varieties and Hyper-K\"ahler varieties 
are building blocks of smooth projective varieties with numerically trivial canonical divisors.
From this point of view, Lesieutre and Satriano proved:

\begin{thm}[{\cite[Corollary 1.5]{LS21}}]\label{thm:CY-to-all}
Let $n$ be a positive integer. Then \cref{a=dconj} is true for all automorphisms of smooth projective varieties $X$ over $\QQ$ 
with dimension at most $n$ and $K_X$ numerically trivial if and only if \cref{a=dconj} is true 
for all automorphisms of smooth Calabi--Yau varieties over $\QQ$ with dimension at most $n$.
\end{thm}

See \cite[Proposition 3.1]{LM21} for a generalization of this theorem to singular varieties.
See \cite{Li22} as well for a related result.

For a smooth projective variety $X$, let us write the irregularity of $X$ by $q(X)$, i.e.\ $q(X) = \dim H^{1}(X, \O_{X}) $.
Chen, Lin, and Oguiso studied Kawaguchi-Silverman conjecture and existence of Zariski dense orbits
for varieties with positive irregularities $q(X) > 0$ in \cite{CLO}.
In any dimension, using albanese morphism and geometry on abelian varieties, they proved:

\begin{thm}[{\cite[Theorem 1.5]{CLO}}]\label{thm:ksc-irregular}
Let $X$ be a smooth projective variety over $\QQ$ with Kodaira dimension zero.
\begin{enumerate}
\item If $q(X) = \dim X$, then \cref{a=dconj} holds for all dominant rational self-maps $f \colon X \dashrightarrow X$.
\item If $q(X) = \dim X -1$, then any birational map $f \colon X \dashrightarrow X$ has no Zariski dense orbits and hence \cref{a=dconj}
is vacuously true.
\end{enumerate}
\end{thm}

See \cite{Li22} for a related result.

\begin{rmk}
For a smooth projective variety $X$ with Kodaira dimension zero,
the irregularity is at most $\dim X$: $0 \leq q(X) \leq \dim X$.
Moreover, if $q(X) = \dim X$, then the albanese morphism is birational and hence $X$ is birational to an abelian variety. 
See \cite{Kaw81}.
\end{rmk}

Constructing a fibration is the most common way to reduce Kawaguchi-Silverman conjecture
to easier cases.
Let $f \colon X \longrightarrow X$ be a surjective self-morphism.
Suppose we have constructed a fibration (i.e.\ proper surjective morphism with connected fibers)
$\pi \colon X \longrightarrow Y$ and a self-morphism $g$ on $Y$ such that $g \circ \pi = \pi \circ f$.
Such an equivariant fibration structure is helpful to understand the dynamics of $f$.
Indeed, if $\d_{1}(f) = \d_{1}(g)$, then Kawaguchi-Silverman conjecture for $f$ follows from that of $g$
(cf.\ \cref{prop:ad-and-semiconj} and comments after that).
But when $\d_{1}(f) > \d_{1}(g)$, there is no way to deduce the conjecture to $g$.
(Case TIR (\cref{CaseTIR}) can be understood as an appearance of this difficulty.)
Projective bundles are probably the easiest fibration and this case is already non-trivial in general.

\begin{thm}[{\cite[Theorem 1.10, Corollary 6.8.]{LS21}}]\ 

\begin{enumerate}
\item
\cref{a=dconj} holds for all surjective self-morphisms on $n$-fold rational normal scrolls over $\QQ$. 

\item
Let $C$ be a smooth projective curve over $\QQ$ and $\E$ be a vector bundle on $C$.
Then \cref{a=dconj} holds for all surjective self-morphisms on the projective bundle $\P_{C}(\E)$ if 
it holds for all surjective self-morphisms on $\P_{C}(\E)$ with $\E$ semistable of degree $0$.
\end{enumerate}
\end{thm}

\begin{thm}[{\cite[Proposition 4.1]{LM21}}]
Let $Y$ be a smooth Fano variety over $\QQ$ of Picard number one. 
Let $X = \P_Y(\E)$ be a projective bundle over $Y$. 
Then \cref{a=dconj} holds for all surjective self-morphisms $f : X \longrightarrow X$.
\end{thm}

We end this subsection with a list of results on $3$-folds.

\begin{thm}\ 

\begin{enumerate}
\item
{\rm \cite[Proposition 1.7]{LS21}}
Let $X$ be a smooth projective variety over $\QQ$ with $\dim X = 3$ and Kodaira dimension zero.
Then \cref{a=dconj} is true for surjective self-morphisms $f \colon X \longrightarrow X$ with $\deg f \geq 2$.

\item
{\rm \cite[Theorem 1.8]{LS21}}
Let $X$ be a smooth Calabi-Yau three fold and $f \colon X \longrightarrow X$ an automorphism.
Suppose either 
\begin{itemize}
\item the second Chern class $c_{2}(X)$ is strictly positive on $\Nef(X)$, or
\item there is a non-zero semi-ample nef divisor $D$ such that $(c_{2}(X)\cdot D)=0$.
\footnote{Oguiso pointed out that the proof contains an error. The proof employs Oguiso's theorem on the finiteness of
$c_{2}$-contractions, but it is only up to automorphisms and therefore one cannot deduce that $D$ is fixed
by some iterates of $f$. Although not yet published, Oguiso has provided an alternative proof for this case and has also proven
some generalizations \cite{OguisoPreprint}.}
\end{itemize}
Then \cref{a=dconj} holds for $f$.

\item
{\rm \cite[Theorem 1.10 (1)]{LS21}}
Let $X$ be a projective variety over $\QQ$ with $\dim X=3$.
Suppose $X$ has a Mori fiber space structure (see \cite[Definition 6.1]{LS21} for the definition).
Then \cref{a=dconj} holds for automorphisms on $X$.

\item
{\rm \cite[Theorem 1.6]{CLO}}
Let $X$ be a smooth projective variety over $\QQ$ with $\dim X = 3$ and $q(X) > 0$.
\begin{enumerate}
\item[ {\rm (a)}] \cref{a=dconj} is true for automorphisms on $X$.
\item[ {\rm (b)}] Suppose $X$ is not covered by rational surfaces. Then \cref{a=dconj} is true for birational self-maps on $X$.
\end{enumerate}

\end{enumerate}
\end{thm}

\subsection{Some useful facts}

In this subsection, we introduce some basic facts
that can be used to construct fibrations equivariant under self-morphisms.
For more about this topic, you may refer to, for instance,
\cite[Theorem 4.19]{Nak10}, \cite[section 4, Lemma 6.2]{MZ18}, \cite[Theorem 7.2]{MZ22}, \cite[section 4]{JSXZ21}.
Let us start with the following fundamental fact.
\begin{prop}
Let $f \colon X \longrightarrow X$ be a surjective self-morphism on a normal projective variety $X$
defined over $\QQ$.
Let $D$ be a Cartier divisor on $X$ such that
\begin{itemize}
\item $f^{*}D \equiv dD$ for some $d \in \Z_{>0}$;
\item $D$ is base point free and (its sublinear system) defines a morphism $\f \colon X \longrightarrow \P^{N}_{\QQ}$.
\end{itemize}
Let $X \xrightarrow{\pi} Y \to \P^{N}_{\QQ}$ be the Stein factorization of $\f$.
Then $f$ induces a surjective self-morphism $g \colon Y \longrightarrow Y$ such that
 \[
\xymatrix{
X \ar[r]^{f} \ar[d]_{\pi} & X \ar[d]^{\pi} \\
Y \ar[r]_{g}& Y
}
\]
commutes.
\end{prop}
\begin{proof}
An irreducible curve $C \subset X$ is mapped to a point by $\pi$ if and only if $(D \cdot C)=0$. 
Therefore $f$ preserves the fibers of $\pi$. Use the Zariski main theorem to conclude the proof.
\end{proof}

Every surjective self-morphism on a projective bundle descends to a
surjective self-morphism on the base after replacing with some iterate.

\begin{prop}
Let $Y$ be a normal projective variety over $\QQ$ and 
let $\E$ be a locally free $\O_{Y}$-module of finite rank.
Let $X = \P(\E) := \Proj \oplus_{n \geq 0}\Sym^{n}\E$
be the projective bundle over $Y$ with structure morphism $\pi \colon X \longrightarrow Y$. 
Let $f \colon X \longrightarrow X$ be a surjective self-morphism.
Then there is a positive integer $m$ and a surjective self-morphism $g \colon Y \longrightarrow Y$
such that
 \[
\xymatrix{
X \ar[r]^{f^{m}} \ar[d]_{\pi}& X \ar[d]^{\pi} \\
Y \ar[r]_{g}& Y
}
\]
commutes.
\end{prop}
\begin{proof}
Use the same argument before \cite[Theorem 2]{Am03}.
\end{proof}

Albanese morphism is compatible with self-morphisms because of the universal property.

\begin{thm}[{\cite[Proposition 3.7]{LM21},\cite[Proposition 5.1]{CLO}}]\label{thm:alb-surj}
Let $X$ be a normal projective variety over $\QQ$.
Let $ \alpha \colon X \longrightarrow A$ be the albanese morphism.

\begin{enumerate}
\item
Let $f \colon X \longrightarrow X$ be a self-morphism.
Then $f$ induces a self-morphism $g \colon A \longrightarrow A$ such that 
$ \alpha  \circ f = g \circ \alpha $.
Moreover, if $f$ has Zariski dense orbit, then $ \alpha$ is surjective.

\item
Suppose $X$ is smooth.
Let $f \colon X \dashrightarrow X$ be a rational map.
Then $f$ induces a self-morphism $g \colon A \longrightarrow A$ such that 
$ \alpha  \circ f = g \circ \alpha $.
Moreover, if $f$ has Zariski dense orbit, then $ \alpha$ is surjective.
\end{enumerate}
\end{thm}

\begin{rmk}
To deduce the last part of \cref{thm:alb-surj}(2) from \cite[Proposition 5.1]{CLO},
we have to confirm that a Zariski dense set of $\QQ$-points is Zariski dense in $X_{\C}$.
We can prove this by using, for example, the openness of the morphism $X_{\C} \longrightarrow X$,
which is because the morphism $\Spec \C \longrightarrow \Spec \QQ$ is universally open 
(cf.\ \cite[\href{https://stacks.math.columbia.edu/tag/0383}{Tag 0383}]{stacks-project}).
\end{rmk}

\section{Toward rational maps}

We have seen that there are many partial results and potential strategies to prove Kawaguchi-Silverman conjecture 
for surjective self-morphisms on projective varieties. 
However, the situation for rational maps is markedly different. 
In certain special cases, there are nice canonical height functions as discussed in \cref{sec:canonical-height}. 
Yet, it is unclear if we can expect the existence of such nice functions for general rational maps. 
The effectiveness of algebro-geometric methods arises from the fact that height functions have functorial property with respect to morphisms. 
When our self-map is merely a rational map, and not a morphism defined everywhere on a projective variety, 
we usually cannot interpret geometric relationships between divisors directly as equations of height functions. 
There are several highly non-trivial results for rational maps that confirm the conjecture 
e.g.\ \cref{thm:ksc-irregular,thm:endo-on-A2-JW,thm:reg-affine-auto,thm:ksc-monomial-maps}.
In this section, we introduce some recent works concerning rational maps.

\subsection{Self-morphisms on affine surfaces}

Among rational maps, those which can be regarded as self-morphisms on quasi-projective varieties are relatively more tractable.
In \cite{JSXZ21}, Jia, Shibata, Xie, and Zhang studied Kawaguchi-Silverman conjecture and Zariski dense orbit conjecture 
for surjective self-morphisms on quasi-projective varieties.
Studying the structure of self-maps and the varieties, they proved:

\begin{thm}[{\cite[Theorem 1.8, Theorem 1.9]{JSXZ21}}]
Let $X$ be a smooth affine surface over $\QQ$ and let $f \colon X \longrightarrow X$ be a finite surjective self-morphism.
Let $ \overline{\kappa}(X)$ be the log Kodaira dimension.
\begin{enumerate}
\item If $\overline{\kappa}(X) \geq 0$ and $\deg f \geq 2$, then \cref{a=dconj} holds for $f$.
\item If $\overline{\kappa}(X) = -\infty$ and $\deg f = 1$, then \cref{a=dconj} holds for $f$.
\item If the \'etale fundamental group $\pi_{1}^{et}(X)$ is infinite, $\deg f \geq 2$, and $X \not\simeq \A^{1} \times {\mathbb{G}}_{{\rm m}}$,
then \cref{a=dconj} holds for $f$.
\end{enumerate}
\end{thm}

\subsection{Cohomologically hyperbolic maps}

Recently, Wang extended the idea of the proof of positivity of canonical height functions for automorphisms on surfaces
and adapt it to a class of rational maps called cohomologically hyperbolic maps \cite{Wa22}.
In this generalization, he did not take the approach to construct a canonical height function, 
and only proved that the height sequence "$h_{H}(f^{n}(x))$" satisfies a certain inequality.

\begin{defn}\label{def:coh-hyp}
Let $X$ be a projective variety over $\QQ$.
Let $p \in \Z$ with $1 \leq p \leq \dim X$.
A dominant rational map $f \colon X \dashrightarrow X$ is said to be $p$-cohomologically hyperbolic if 
$\d_{p}(f) > \d_{i}(f)$ for all $i \in \{1, \dots ,\dim X\} \setminus \{p\}$.
\end{defn}

This class of rational maps were first introduced by Guedj \cite{Gu10} as a cohomological analogue of hyperbolicity.
Wang proved:

\begin{thm}[{\cite[Theorem 1.4]{Wa22}}]\label{thm:wang}
Let $X$ be a smooth projective surface over $\QQ$.
Let $f \colon X \dashrightarrow X$ be a $1$-cohomologically hyperbolic dominant rational map (i.e. $\d_{1}(f) > \d_{2}(f)$).
Assume there is an open dense subset $W \subset X \setminus I_{f}$ with $f(W) \subset W$ such that either 
\begin{enumerate}
\item
$f|_{W} \colon W \longrightarrow W$ is \'etale, or

\item
$W = \A^{2}_{\QQ}$.

\end{enumerate}
Then \cref{a=dconj} holds for $f$.
\end{thm}

\begin{rmk}
In \cite[Theorem 1.4]{Wa22}, $f|_{W} \colon W \longrightarrow W$ is assumed to be surjective,
but it is not necessary.
\end{rmk}

\begin{rmk}\label{rmk:jwx-vs-wang}
The second case (i.e.\ affine space case) also follows from \cref{rmk:JWX}.
Jonsson-Wulcan-Xie's proof uses valuative trees as a key ingredient to prove the positivity of canonical height functions.
Wang's approach is completely different, but it actually uses valuative tree indirectly as the proof requires
dynamical Mordell-Lang conjecture for $\A^{2}_{\QQ}$, which is proven by Xie using valuative tree \cite{Xi17}.
\end{rmk}

Wang needs the assumptions (1)(2) in \cref{thm:wang} entirely because 
he uses dynamical Mordell-Lang conjecture.
If you assume that the orbit is generic and not merely Zariski dense,
then you can avoid this difficulty.
Recall that an orbit $O_{f}(x)$ is generic if it is infinite and its intersection with any proper closed subset
is finite. We remark that dynamical Mordell-Lang conjecture implies every Zariski dense orbit is generic.
In our joint work, Wang and the author proved the following higher dimensional generalization of \cref{thm:wang}.

\begin{thm}[{\cite[Theorem A, Corollary 1.8]{MW22}}]\label{thm:cohhyp}
Let $X$ be a smooth projective variety over $\QQ$.
Let $f \colon X \dashrightarrow X$ be a $p$-cohomologically hyperbolic map.
Let $x \in X_{f}(\QQ)$ be a point with generic $f$-orbit.
Then we have
\begin{align}
\underline{\alpha}_{f}(x) \geq \frac{\d_{p}(f)}{\d_{p-1}(f)}.
\end{align}
If $p=1$, then $ \alpha_{f}(x)$ exists and $ \alpha_{f}(x)= \d_{1}(f)$.
\end{thm}

Note that if $f$ is a birational self-map on a smooth projective surface,
then $1$-cohomological hyperbolicity is equivalent to $\d_{1}(f) > 1$.
For such a case, \cref{a=dconj} with the stronger assumption that $O_{f}(x)$ is generic
follows from \cref{thm:cohhyp}.

\begin{rmk}\label{rmk:cohhyp-genorb}
As an application of some properties of arithmetic degrees and dynamical degrees,
we proved that $1$-cohomologically hyperbolic map has points with generic orbits \cite[Theorem B]{MW22}.
\end{rmk}

\section{Related topics}

There are several works concerning arithmetic degrees.

Kawaguchi-Silverman conjecture asks if a point with dense orbit 
has arithmetic degree equal to dynamical degree, but it is natural to ask
if there exists at least one point $x$ that satisfies the equality $ \alpha_{f}(x) = \d_{1}(f)$.
This is proven for surjective self-morphisms on projective varieties \cite[Lemma 9.1,Corollary 9.3]{MSS18a}
(in this paper, the variety is assumed smooth, but the same proof works for singular varieties).
By the following works by Sano and Shibata, it is proven that such points are Zariski dense over $\QQ$
and over number fields under some additional assumptions \cite{SS23,SS21}.
For rational maps, it is not known in general that there is at least one point $x$ with $ \alpha_{f}(x) = \d_{1}(f)$.
There are some results to this direction \cite[Theorem 3]{KS14}, \cite[Theorem B, Proposition 4.6]{MW22}.

There are several works studying which values arithmetic degrees can take.
In general, the set $\{ \alpha_{f}(x) \mid x \in X_{f}(\QQ) \}$ of arithmetic degrees can be infinite (\cref{thm:inf-set-of-ad}), 
but for surjective self-morphisms on normal projective varieties, it is finite and described by the eigenvalues of
$f^{*} \colon N^{1}(X) \longrightarrow N^{1}(X)$ (\cref{thm:ad-of-mor}).
Lin determined the set of arithmetic degrees of monomial maps on projective toric varieties \cite[Theorem A]{Li19}.
In \cite{MS20}, we determined the set of arithmetic degrees of surjective self-morphisms on semi-abelian varieties.
For a surjective self-morphism on a normal projective variety $f \colon X \longrightarrow X$,
members of the set $ \{ |\mu| \mid \text{$\mu$ is an eigenvalue of $f^{*}$ with $|\mu|>1$} \}$ are candidates of 
arithmetic degrees. In \cite{Na22}, Nasserden studies the difference of this set and the set of arithmetic degrees ($\neq 1$).
For surjective self-morphisms on $\Q$-factorial toric varieties,
he proved that all such values can be realized as arithmetic degrees.
Also he constructed examples of self-morphisms on abelian varieties such that 
some of such candidate values are not realized as arithmetic degrees.

For surjective self-morphism on a projective variety,
we can observe that the set of points $x$ with $ \alpha_{f}(x) < \d_{1}(f)$ is small.
More precisely, it seems that the set of points $x$ with $ \alpha_{f}(x) < \d_{1}(f)$ and defined over number fields of bounded degree 
is not Zariski dense.
This is stated as ``small Arithmetic Non-density conjecture (sAND)'' in \cite[Conjecture 1.3]{MMSZ23}
and confirmed for several cases. See \cite{Shi19} and \cite[Theorem 1.7]{Na22} for related results.
Note that sAND implies Kawaguchi-Silverman conjecture.
It is also worth mentioning that we cannot expect the same statement for rational maps \cite[Remark 1.4(6)]{MMSZ23}.

Kawaguchi-Silverman conjecture compares arithmetic degree $ \alpha_{f}(x)$ 
with the first dynamical degree $\d_{1}(f)$ of $f$.
It is natural to ask what the arithmetic counter part of higher dynamical degrees $\d_{p}(f)$ is.
Dang, Ghioca, Hu, Lesieutre, and Satriano defined higher arithmetic degrees using Arakelov theory 
and extend Kawaguchi-Silverman conjecture to that setting \cite{DGHLS22}.
In their formulation, higher arithmetic degrees are defined for subvarieties (not points)
and measures arithmetic complexity of orbits of subvarieties.
Very recently, Jiarui Song established basic properties of higher arithmetic degrees
and gave a conceptual proof of the inequality $ \overline{\alpha}_{f}(x) \leq \d_{1}(f)$
for arbitrary dominant rational maps on normal projective variety over a number field \cite{So23}.
Moreover, he gave a counter example to DGHLS's higher dimensional version of Kawaguchi-Silverman conjecture
using the morphism that produces a counter example to the original dynamical Manin-Mumford conjecture.

\section{Future directions}

Let us conclude this survey by discussing potential future directions.

The most challenging and crucial case of the Kawaguchi-Silverman conjecture might be the dominant rational maps on the projective space $\P^{N}_{\QQ}$. 
The projective space has densely many rational points and a rich structure of self-maps. 
Another demanding case is self-maps on Calabi-Yau varieties. 
In broad terms, Calabi-Yau varieties are one of the building blocks of varieties of Kodaira dimension zero, 
alongside Abelian varieties and Hyper-K\"ahler varieties. 
The Kawaguchi-Silverman conjecture is fully resolved for Abelian varieties, 
and for Hyper-K\"ahler varieties, it is solved for all surjective self-morphisms. 
However, we lack such general results for Calabi-Yau varieties and even the solution for automorphism case would be
very interesting (See \cref{thm:CY-to-all}). 
It's worth mentioning that Calabi-Yau varieties are also expected to have many rational points and exhibit a rich dynamical structure. 

Solving the Kawaguchi-Silverman conjecture in full generality may be quite difficult, 
but exploring, for instance, the cohomologically hyperbolic case could be intriguing. 
It would also be interesting to examine how prevalent cohomologically hyperbolic maps are. 
Given that they are defined by the non-coincidence of successive dynamical degrees, 
it is reasonable to anticipate the existence of many such maps, particularly on specific varieties such as projective spaces.

The Case TIR (\cref{CaseTIR}), which emerged during the study of the Kawaguchi-Silverman conjecture, 
is interesting from the viewpoint of the structure theory of surjective self-morphisms on projective varieties as well. 
It is predicted that the Case TIR does not occur, but this remains an open question in general to the best of my knowledge.

In addition to proving the equality $ \alpha_{f}(x) = \d_{1}(f)$, 
establishing the existence of the arithmetic degree for arbitrary dominant rational self-maps presents a significant challenge. 
This is tied to a deeper understanding of the sequence $h_{H}(f^{n}(x))$. 
For surjective self-morphisms on projective varieties, this sequence is well-understood, 
at least when $ \alpha_{f}(x) > 1$ (See \cite[Theorem 1.1]{Sa18}). 
For rational maps, this sequence behaves wildly (See \cite[Example 3.1]{Li19}), making it intriguing to discern some patterns.

Investigating the function field analogue of the Kawaguchi-Silverman conjecture could also be of interest. 
When the coefficient field of the function field is infinite, some additional assumptions are necessary to avoid ``constant points", 
but it seems reasonable to expect that the same essentially holds. 
Refer to \cref{rmk:other-fld} and \cite{MSS18b,Dim15,Xi23} for more on this direction.

As demonstrated in \cite{JSXZ21,MW22}, the study of arithmetic degrees has applications to the Zariski dense orbit conjecture. 
The arithmetic degree serves as a measure of the complexity of the orbits, meaning that, 
under certain additional assumptions on the geometric structure of the self-maps, 
points with maximal arithmetic degrees have dense orbits. 
Pursuing this direction and discovering further applications to the Zariski dense orbit conjecture would be highly fascinating.

\begin{ack}
The author would like to thank 
Mattias Jonsson, John Lesieutre, Sheng Meng, Keiji Oguiso, Kaoru Sano, Matt Satriano, Takahiro Shibata, Joe Silverman, Elizabeth Wulcan, Junyi Xie, 
Shou Yoshikawa and D.Q. Zhang
for valuable comments on the manuscript of this survey and answering his questions.
Some of the works related to the equivariant MMP referred in this survey started during the Simons symposium 
``Algebraic, Complex and Arithmetic Dynamics (2019)''.
The author would like to thank Simons foundation and the organizers, Laura DeMarco and Mattias Jonsson,
for  providing a stimulating and collaborative environment.
\end{ack}

\bibliographystyle{abbrv}

\bibliography{KSC}

\end{document}